\documentclass{amsart}
\usepackage[dvipsnames]{xcolor}
\usepackage{amsfonts, amsthm, amsmath, amssymb, enumerate, mathrsfs, tikz, tikz-cd, ifthen,pgfplots, chngcntr, mathtools, graphicx,multicol,array, bbm, enumitem}

\usepackage{soul}

\usepackage{subcaption}
\usepackage{hyperref}
	\hypersetup{colorlinks=true, linkcolor={blue}, urlcolor={blue}, citecolor={blue}}
\usepackage[toc,page]{appendix}
\usetikzlibrary{arrows}
\usetikzlibrary{decorations.markings}

\newtheorem{theorem}[subsection]{Theorem}
\newtheorem{prop}[subsection]{Proposition}
\newtheorem{lemma}[subsection]{Lemma}

\newtheorem{corollary}[subsection]{Corollary}
\theoremstyle{definition}
\newtheorem{definition}[subsection]{Definition}
\newtheorem{example}[subsection]{Example}
\newtheorem{notation}[subsection]{Notation}
\newtheorem{rmk}[subsection]{Remark}
\newtheorem{warning}[subsection]{Warning}

\DeclareMathOperator{\im}{im}

\newcommand{\M}{\mathbb{M}}
\newcommand{\R}{\mathbb{R}}
\newcommand{\Z}{\mathbb{Z}}
\renewcommand{\H}{\tilde{H}}
\renewcommand{\dot}{\small{$\bullet$}}

\counterwithin{equation}{subsection}

\newcommand{\cone}[3]{
	\draw[thick, #3] (#1+1/2,5) -- (#1+1/2,#2+1/2) -- (5-#2+#1,5);
	\draw[thick, #3] (#1+1/2,-5) -- (#1 +1/2, #2 -1.5) -- (-3-#2+#1,-5)
}
\newcommand{\anti}[3]{
	\draw[thick, #3] (#1+1/2, -5) -- (#1+1/2, 5);
	\draw[thick, #3] (#1+#2+1/2, -5) -- (#1+#2+1/2, 5);
	{\ifthenelse{#2>1}{
	\foreach \y in {-5,...,2} {
	\draw[thick, #3] (#1+1/2, \y +1/2) -- (#1+#2+1/2, \y + #2+1/2);
	}
	}{
	\foreach \y in {-5,-4,-3,-2,-1,0,1,2,3} {\draw[thick, #3] (#1+1/2, \y +1/2) -- (#1+#2+1/2, \y + #2+1/2);}}
	};
}

\newcommand{\lab}[3]{
\draw[#3] (#1+1/2,5.5) node{\tiny{$#2$}}
}

\begin{document}

\title[Equivariant fundamental classes]{Equivariant fundamental classes in $RO(C_2)$-graded cohomology with $\underline{\Z/2}$-coefficients}

\begin{abstract} Let $C_2$ denote the cyclic group of order two. Given a manifold with a $C_2$-action, we can consider its equivariant Bredon $RO(C_2)$-graded cohomology. In this paper, we develop a theory of fundamental classes for equivariant submanifolds in $RO(C_2)$-graded cohomology with constant $\Z/2$ coefficients. We show the cohomology of any $C_2$-surface is generated by fundamental classes, and these classes can be used to easily compute the ring structure. To define fundamental classes we are led to study the cohomology of Thom spaces of equivariant vector bundles. In general the cohomology of the Thom space is not just a shift of the cohomology of the base space, but we show there are still elements that act as Thom classes, and cupping with these classes gives an isomorphism within a certain range. 
\end{abstract}

\author[C. Hazel]{Christy Hazel}

{\let\thefootnote\relax\footnote{{The author was partially supported by NSF grant DMS-1560783.}}}

\maketitle

\tableofcontents

\section{Introduction}\label{ch:intro}
There has been recent interest in better understanding $RO(C_2)$-graded Bredon cohomology. For example, explicit computations can be found in \cite{CHT}, \cite{Ha}, \cite{Ho}, \cite{K2}, \cite{LFdS}, and \cite{Sh}, and certain freeness and structure theorems can be found in \cite{K1} and \cite{M1}. In this paper, we consider $C_2$-manifolds and develop a theory of fundamental classes for equivariant submanifolds in constant $\Z/2$ coefficients. These classes can be defined for both free and nonfree submanifolds.  When two submanifolds intersect transversally, the cup product of their classes is given in terms of the fundamental class of their intersection. In order to define these classes, we consider equivariant Thom spaces for real $C_2$-vector bundles and prove properties of the $RO(C_2)$-graded cohomology of these spaces with constant $\Z/2$ coefficients.

To say more about the cohomology of Thom spaces, let $\M_2$ denote the cohomology of a point with constant $\Z/2$ coefficients, and note the cohomology of any $C_2$-space with $\underline{\Z/2}$-coefficients is an $\M_2$-module. We will show when the base is nonfree there is a unique class in a predicted grading that generates a free submodule and forgets to the singular Thom class. Cupping with this class gives an isomorphism from the cohomology of the base space to the cohomology of the Thom space within a certain range. When the base is free, we show there are infinitely many classes connected by a predicted module action that forget to the singular Thom class, and cupping with any one of these classes produces an isomorphism from the shifted cohomology of the base space to the cohomology of the Thom space. 

The Thom classes for normal bundles allow us to define the previously mentioned fundamental classes. For nonfree submanifolds, there will be a unique fundamental class whose grading is determined by the normal bundle. For free submanifolds, there will be infinitely many fundamental classes all connected by a predicted module action. We will show the cohomology of any $C_2$-surface is generated by fundamental classes of submanifolds, and that we can use these classes to easily compute the ring structure. 
\medskip

Let us begin by recalling some facts about Bredon cohomology; a more detailed exposition can be found in \cite[Section 2]{HHR}, for example. For a finite group $G$ the Bredon cohomology of a $G$-space is a sequence of abelian groups graded on $RO(G)$, the Grothendieck group of finite-dimensional, real, orthogonal $G$-representations. When $G$ is the cyclic group of order two, recall any $C_2$-representation is isomorphic to a direct sum of trivial representations and sign representations. Thus $RO(C_2)$ is a free abelian group of rank two, and the Bredon cohomology of any $C_2$-space can be regarded as a bigraded abelian group. 

We will use the motivic notation $H^{*,*}(X;M)$ for the Bredon cohomology of a $C_2$-space $X$ with coefficients in a Mackey functor $M$. The first grading indicates the dimension of the representation, and we will often refer to this grading as the ``topological dimension". The second grading indicates the number of sign representations appearing and will be referred to as the ``weight". Given any $C_2$-space $X$, there is always an equivariant map $X\to pt$ where $pt$ denotes a single point with the trivial $C_2$-action. This gives a map of bigraded abelian groups $H^{*,*}(pt;M)\to H^{*,*}(X;M)$. If $M$ has the additional structure of a Green functor, then in fact this is a map of bigraded rings, and thus $H^{*,*}(X;M)$ forms a bigraded algebra over the cohomology of a point. In this paper, we will be working with the constant Mackey functor $\underline{\Z/2}$, which is a Green functor, and we will write $\M_2$ for the bigraded ring $H^{*,*}(pt;\underline{\Z/2})$. In addition to being a Green functor, $\underline{\Z/2}$ satisfies an additional property ($\text{tr}(1)=2$) that ensures $H^{*,*}(X;\underline{\Z/2})$ is a bigraded commutative ring.

\subsection{The cohomology of orbits} 
Since the cohomology of any $C_2$-space with $\underline{\Z/2}$-coefficients is a bigraded module over a bigraded ring, we can use a grid to record information about the cohomology groups and module structures. The ring $\M_2=H^{*,*}(pt;\underline{\Z/2})$ is illustrated on the left-hand grid in Figure \ref{fig:m2intro} below. Each dot represents a copy of $\Z/2$, and the connecting lines indicate properties of the ring structure. For example, the portion of the picture in quadrant one is polynomial in two elements $\rho$ and $\tau$ that have bidegrees $(1,1)$ and $(0,1)$, respectively. The diagonal lines in the picture indicate multiplication by $\rho$, while the vertical lines indicate multiplication by $\tau$. A full description of $\M_2$ can be found in Section \ref{ch:backcoh}. In practice, it is often easier to work with the abbreviated picture shown on the right-hand grid. 

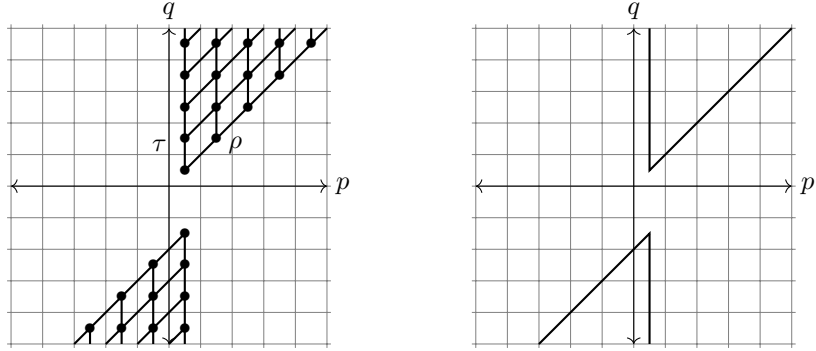
\begin{figure}[ht]
	\begin{tikzpicture}[scale=.42]
		\draw[help lines,gray] (-5.125,-5.125) grid (5.125, 5.125);
		\draw[<->] (-5,0)--(5,0)node[right]{$p$};
		\draw[<->] (0,-5)--(0,5)node[above]{$q$};
		\cone{0}{0}{black};
		\foreach \y in {0,1,2,3,4}
			\draw (0.5,\y+.5) node{\small{$\bullet$}};
		\foreach \y in {1,2,3,4}
			\draw (1.5,\y+.5) node{\small{$\bullet$}};
		\foreach \y in {2,3,4}
			\draw (2.5,\y+.5) node{\small{$\bullet$}};
		\foreach \y in {3,4}
			\draw (3.5,\y+.5) node{\small{$\bullet$}};
		\foreach \y in {4}
			\draw (4.5,\y+.5) node{\small{$\bullet$}};
		\foreach \y in {-1,-2,-3,-4}
			\draw (.5,\y-.5) node{\small{$\bullet$}};
		\foreach \y in {-2,-3,-4}
			\draw (-.5,\y-.5) node{\small{$\bullet$}};
		\foreach \y in {-3,-4}
			\draw (-1.5,\y-.5) node{\small{$\bullet$}};
		\foreach \y in {-4}
			\draw (-2.5,\y-.5) node{\small{$\bullet$}};
		\draw[thick](1.5,5)--(1.5,1.5);
		\draw[thick](2.5,2.5)--(2.5,5);
		\draw[thick](3.5,3.5)--(3.5,5);
		\draw[thick](4.5,4.5)--(4.5,5);
		\draw[thick](.5,1.5)--(4,5);
		\draw[thick](.5,2.5)--(3,5);
		\draw[thick](.5,3.5)--(2,5);
		\draw[thick](.5,4.5)--(1,5);
		\draw[thick](-.5,-2.5)--(-.5,-5);
		\draw[thick](-1.5,-3.5)--(-1.5,-5);
		\draw[thick](-2.5,-4.5)--(-2.5,-5);
		\draw[thick](.5,-2.5)--(-2,-5);
		\draw[thick](.5,-3.5)--(-1,-5);
		\draw[thick](.5,-4.5)--(0,-5);
		\draw (2.1,1.3) node{$\rho$};
		\draw (-.3,1.3) node{$\tau$};
	\end{tikzpicture}
	\hspace{0.5in}
	\begin{tikzpicture}[scale=.42]
		\draw[help lines,gray] (-5.125,-5.125) grid (5.125, 5.125);
		\draw[<->] (-5,0)--(5,0)node[right]{$p$};
		\draw[<->] (0,-5)--(0,5)node[above]{$q$};
		\cone{0}{0}{black};
	\end{tikzpicture}
	\caption{The ring $\M_2=H^{*,*}(pt;\underline{\Z/2}$) with $H^{p,q}(pt)$ in spot $(p,q)$.}
	\label{fig:m2intro}
\end{figure}

For the examples that follow, we also need the cohomology of the free orbit $C_2$. We denote this $\M_2$-module by $A_0=H^{*,*}(C_2;\underline{\Z/2})$. Algebraically, $A_0\cong \tau^{-1}\M_2/(\rho)\cong \Z/2[\tau,\tau^{-1}]$, and in the figure below, we provide a detailed drawing of this module on the left and an abbreviated drawing on the right. The dots again indicate copies of $\Z/2$ while the vertical lines indicate action by $\tau\in H^{0,1}(pt;\underline{\Z/2})$.

\begin{figure}[ht]
	\begin{tikzpicture}[scale=.42]
		\draw[help lines,gray] (-2.125,-5.125) grid (2.125, 5.125);
		\draw[<->] (-2,0)--(2,0)node[right]{$p$};
		\draw[<->] (0,-5)--(0,5)node[above]{$q$};
		\anti{0}{0}{black};
		\foreach \y in {-5,...,4}
			\draw (0.5,\y+.5) node{\small{$\bullet$}};
	\end{tikzpicture}\hspace{0.5 in}
	\begin{tikzpicture}[scale=.42]
		\draw[help lines,gray] (-2.125,-5.125) grid (2.125, 5.125);
		\draw[<->] (-2,0)--(2,0)node[right]{$p$};
		\draw[<->] (0,-5)--(0,5)node[above]{$q$};
		\anti{0}{0}{black};
	\end{tikzpicture}
	\caption{The $\M_2$-module $A_0=H^{*,*}(C_2;\underline{\Z/2})$.}
\end{figure}
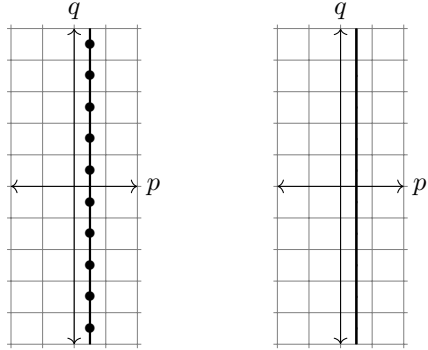

\subsection{Summary of main points}
Now that we have established some notation, we provide a brief summary of the main goals of this paper. There are two main topics, fundamental classes and the cohomology of Thom spaces. These topics can be further divided into five subtopics:  

\begin{enumerate}[wide, labelwidth=!, labelindent=0pt, itemsep=0.7em]
	\item \emph{Nonfree fundamental classes:} Given a nonfree $C_2$-manifold $X$ and a nonfree $C_2$-submanifold $Y$, we prove there exists a fundamental class $[Y]\in H^{n-k,q}(X;\underline{\Z/2})$ where $k=\dim(Y)$ and $q$ is found as follows. Consider the restriction of the equivariant normal bundle of $Y$ in $X$ to the fixed set $Y^{C_2}$. Over each component of this fixed set, the fibers are $C_2$-representations, and the isomorphism type of the representation is constant over each component. Thus each component corresponds to a representation of some weight, and the integer $q$ is chosen to be the maximum such weight.
	\item \emph{Free fundamental classes:} Given a $C_2$-manifold $X$ (free or nonfree) and a free submanifold $Y$, we show there is an infinite family of classes $[Y]_q\in H^{n-k,q}(X;\underline{\Z/2})$ where $k=\dim(Y)$ and $q$ is any integer. These classes satisfy the module relation $\tau\cdot[Y]_{q}=[Y]_{q+1}$. 
	\item \emph{Intersection product of fundamental classes:} We show the cup product of two fundamental classes for nonfree submanifolds $Y$ and $Z$ that intersect transversally is given by a predicted $\tau$-multiple of the fundamental class of the intersection. For free submanifolds, we show $[Y]_r\smile [Z]_s=[Y\cap Z]_{r+s}$.
	\item \emph{Restricted Thom isomorphism theorem for nonfree spaces:} Let $X$ be a finite, nonfree $C_2$-CW complex and $E\to X$ be an $n$-dimensional $C_2$-vector bundle whose maximum weight representation over $X^{C_2}$ is $q$. We show there is a unique class $u_E\in H^{n,q}(E,E-0;\underline{\Z/2})$ that restricts to $\tau$-multiples of the generators of the cohomology of the fibers. We also show cupping with this class gives an isomorphism $H^{f,g}(X;\underline{\Z/2})\to H^{f+n,g+q}(E,E-0;\underline{\Z/2})$ whenever $g\geq f$.
	\item\emph{Thom isomorphism theorem for free spaces:} Let $X$ be a finite, free $C_2$-CW complex. For an $n$-dimensional $C_2$-vector bundle $E\to X$, we show for each integer $q$ there is a unique class $u_{E,q}\in H^{n,q}(E,E-0;\underline{\Z/2})$ that restricts to the nonzero class in the cohomology of the fibers. In this case, cupping with any one of these Thom classes provides an isomorphism from the cohomology of the base space to the cohomology of a shift of the Thom space. 
\end{enumerate}

\begin{rmk}
These equivariant fundamental classes act in many ways just like their nonequivariant analogs, though, there are still various subtleties that arise in the equivariant context. For example, one might expect nonfree fundamental classes to always generate free summands and free fundamental classes to string together to form modules similar to $A_0$. While this is often the case, there are exceptions. As shown in Example \ref{introex3} below, there can be free submanifolds whose classes are nonzero for a while, but become zero in high enough weight. This and other subtleties are best explained through a series of examples. 
\end{rmk}

\subsection{Examples of equivariant fundamental classes} 
Before stating the general results, let's consider some examples of $C_2$-surfaces to determine what properties we might expect of equivariant fundamental classes.
\begin{example}
 We begin with a simple example. Let $X$ denote the $C_2$-space whose underlying space is the genus one torus and whose $C_2$-action is given by a reflection, as shown in Figure \ref{fig:intro1ex}. The fixed set is shown in {\color{blue}{\textbf{blue}}}. As shown in the author's previous work \cite[Theorem 6.6]{Ha}, the cohomology of this space is given by
	\[
	H^{*,*}(X;\underline{\Z/2})\cong \M_2 \oplus \Sigma^{1,0}\M_2 \oplus \Sigma^{1,1}\M_2 \oplus \Sigma^{2,1}\M_2.
	\]
An abbreviated illustration of this module is given below.

\begin{figure}[ht]
	\begin{subfigure}[b]{0.45\textwidth}
	\centering
	\begin{tikzpicture}[scale=1.1]
		\draw[<->] (-.2,1.7)--(.2,1.7);
		\draw (0,0) ellipse (2cm and 1cm);
		\draw (-0.4,0) to[out=330,in=210] (.4,0) ;
		\draw (-.35,-.02) to[out=35,in=145] (.35,-.02);
		\draw[very thick,blue] (0,.55) ellipse (.1cm and 0.45cm);
		\draw[white,fill=white] (-.3,0.15) rectangle (0,.96);
		\draw[very thick,blue,dashed] (0,.55) ellipse (.1cm and 0.45cm);
		\draw[very thick,blue] (0,-.56) ellipse (.1cm and 0.43cm);
		\draw[white,fill=white] (-.3,-0.95) rectangle (0,-.14);
		\draw[very thick,blue,dashed] (0,-.56) ellipse (.1cm and 0.43cm);
		\draw[thick,red] (0,0) ellipse (1cm and 0.5cm);
		\draw[green] (0.1,.5)node{\dot};
		\draw[green] (0.3,.6)node{\small{$a$}};
		\draw[green] (0.1,-.5)node{\dot};
		\draw[green] (0.3,-.6)node{\small{$b$}};
		\draw[blue] (0,1.3) node{\small{$C$}};
		\draw[blue] (0,-1.3) node{\small{$C'$}};
		\draw[red] (-1.25,0) node{\small{$D$}};
	\end{tikzpicture}
	\vspace{0.2in}
	\end{subfigure}
	\begin{subfigure}[b]{0.45\textwidth}
	\centering
	\begin{tikzpicture}[scale=0.41]
		\draw[help lines,gray] (-5.125,-5.125) grid (5.125, 5.125);
		\draw[<->] (-5,0)--(5,0)node[right]{$p$};
		\draw[<->] (0,-5)--(0,5)node[above]{$q$};
		\cone{0}{0}{black};
		\cone{1.2}{1}{black};
		\draw[thick] (1.3,5)--(1.3,.5)--(5,4.2);
		\draw[thick] (1.3,-5)--(1.3, -1.5)--(-2.2,-5);
		\draw[thick] (2.5,5)--(2.5,1.5)--(5,4);
		\draw[thick] (2.5,-5)--(2.5,-.5)--(-2,-5);
	\end{tikzpicture}
	\end{subfigure}
	\caption{The space $X$ and the $\M_2$-module $H^{*,*}(X;\underline{\Z/2})$.}
	\label{fig:intro1ex}
\end{figure}
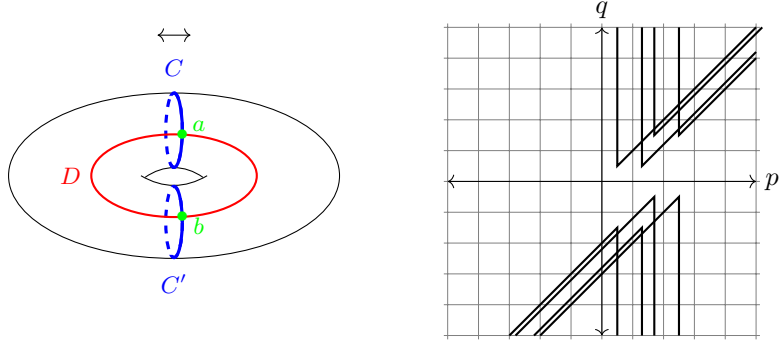

Suppose we wanted to find equivariant submanifolds whose fundamental classes generate these free summands. Take, for example, the fixed circle labeled $C$ above. Topologically, this is a codimension one submanifold, so we would expect to have a fundamental class in bidegree $(1,q)$ for some $q$, but how do we determine this value of $q$? Let's consider a tubular neighborhood of $C$. Note the tubular neighborhood is equivariantly homeomorphic to $C\times \R^{1,1}$ where $\R^{1,1}$ denotes the sign representation, so one might expect a fundamental class in bidegree $(1,1)$. Indeed, we will show there is such a class $[C]\in H^{1,1}(X;\underline{\Z/2})$, and furthermore, this class generates a free summand in bidegree $(1,1)$. We also have a class $[C']\in H^{1,1}(X;\underline{\Z/2})$ corresponding to the other fixed circle, and we will show $[C]=[C']+\rho\cdot 1$. 

Now consider the circle $D$ that travels around the hole of the torus and is isomorphic to a circle with a reflection action. In this case, the tubular neighborhood of $D$ is homeomorphic to $D\times \R^{1,0}$ where $\R^{1,0}$ is the trivial representation. We expect to get a class in bidegree $(1,0)$, and indeed such a class $[D]$ exists and generates a free summand.

These two circles intersect at a single fixed point $a$ whose tubular neighborhood is the unit disk in $\R^{2,1}$. The fundamental class $[a]$ generates a free summand in bidegree $(2,1)$, and furthermore we obtain the relation
	\[
	[C]\smile [D]=[a].
	\]
The submanifolds $C$ and $C'$ do not intersect, so we also obtain the relation
	\[
	[C]\smile [C]=[C]\smile ([C']+\rho\cdot 1)=\rho\cdot [C].
	\]
We conclude
	\[
	H^{*,*}(X;\underline{\Z/2}) \cong \M_2[x,y]/(x^2=\rho x, y^2=0),~~~|x|=(1,1),~|y|=(1,0)
	\]
as an $\M_2$-algebra.
\end{example}

\begin{rmk}
In the example above, all of the considered submanifolds had trivial normal bundles, and this led to an obvious choice of bidegree for each fundamental class. Of course, we would like to have fundamental classes for submanifolds whose normal bundles are nontrivial. The next example illustrates this.
\end{rmk}

\begin{example}
Let $Y$ denote the $C_2$-space given by the projective plane with action induced by the rotation action on the disk as shown in Figure \ref{fig:rp2twintro}. Again the fixed set is shown in {\color{blue}{\textbf{blue}}}. The cohomology of $Y$ as an $\M_2$-module is given by
	\[
	H^{*,*}(Y;\underline{\Z/2})\cong \M_2 \oplus \Sigma^{1,1}\M_2 \oplus \Sigma^{2,1}\M_2.
	\]

\begin{figure}[ht]
	\begin{subfigure}[b]{0.45\textwidth}
	\centering
	\begin{tikzpicture}[scale=1.7]
		\draw[blue,very thick, fill=lightgray, decoration={markings, mark=at position 0 with {\arrow[black,scale=.5]{*}}, mark=at position 0.25 with {\arrow[black]{>}}, mark=at position 0.5 with {\arrow[black,scale=.5]{*}}, mark=at position 0.75 with {\arrow[black]{>}}}, postaction={decorate}] (0,0) ellipse (1 cm and .7 cm);
		\draw[red, very thick] (-.2,.68)--(.2,-.68);
		\draw[blue] (-.2,.68) node{$\bullet$};
		\draw[blue] (-.25,.82) node{$p$};
		\draw[red] (-.3,.5) node{$C'$};
		\draw[blue] (.2,-.68) node{$\bullet$};
		\draw[blue] (.25,-.84) node{$p$};
		\draw[blue] (0,0)node {$\bullet$};
		\draw[->] (.14,1) arc (-40:220:.20cm);
		\draw[blue] (.15,0) node{$q$};
		\draw[blue] (1.1,.15) node{$C$};
	\end{tikzpicture}
	\vspace{.2in}
	\end{subfigure}
	\begin{subfigure}[b]{0.45\textwidth}
	\centering
	\begin{tikzpicture}[scale=.41]
		\draw[help lines,gray] (-3.125,-5.125) grid (6.125, 5.125);
		\draw[<->] (-3,0)--(6,0)node[right]{$p$};
		\draw[<->] (0,-5)--(0,5)node[above]{$q$};
		\cone{-0.1}{0}{black};
		\cone{1.1}{1}{black};
		\cone{2.1}{1}{black};
	\end{tikzpicture}
	\end{subfigure}
	\caption{The space $Y$ and the $\M_2$-module $H^{*,*}(Y;\underline{\Z/2})$.}
	\label{fig:rp2twintro}
\end{figure}
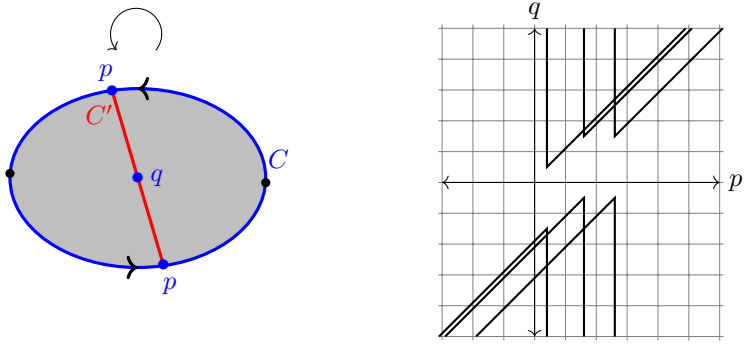

We again look for equivariant submanifolds. Consider the fixed circle labeled $C$. The equivariant normal bundle $N$ is given by the nontrivial M\"obius bundle over $S^1$. This bundle is not trivial, but observe every point $x\in C$ has an equivariant neighborhood $U_x$ such that $N|_{U_x}\cong U_x\times \R^{1,1}$. In other words, this bundle is a locally trivial $\R^{1,1}$-bundle, so we get a class $[C]\in H^{1,1}(X;\underline{\Z/2})$.

Now consider the circle $C'$. This circle is isomorphic to a circle with a reflection action, and again the normal bundle $N'$ is the nontrivial bundle over $S^1$. Around the fixed point $p$ there is a neighborhood $U$ such that $N'|_U\cong U\times \R^{1,0}$ while around the point $q$ there is a neighborhood $V$ such that $N'|_V \cong V \times \R^{1,1}$. The bundle $N'$ is not a locally trivial $W$-bundle for any $C_2$-representation $W$, so if it even exists, the bidegree of the fundamental class for $C'$ is unclear.

We will show the fundamental class $[C']$ does exist, and its grading is determined by the fibers of the normal bundle over the fixed set. In this case, $N'|_p\cong \R^{1,0}$ while $N'|_q\cong \R^{1,1}$. The class $[C']$ will have grading $(1,k)$ where $k$ is the \emph{maximum} weight representation appearing over the fixed set, so in this case $k=1$. 

The classes $[C]$ and $[C']$ are both in bidegree $(1,1)$, and we will see in Example \ref{rp2twex} that $[C]=[C']+\rho\cdot 1$. Let's consider their product. One might hope
	\[
	[C]\smile [C'] = [C\cap C'] = [p],
	\]
but something has gone wrong. The product has bidegree $(2,2)$, while the fundamental class $[p]$ has bidegree $(2,1)$. We will show a slightly modified formula holds. Namely
	\[
	[C]\smile [C'] = \tau\cdot[p],
	\]
where recall $\tau\in H^{0,1}(pt;\underline{\Z/2})$. This gives us the relation 
	\[
	[C]^2=[C]([C']+\rho\cdot1)=\tau\cdot[p]+\rho\cdot[C].
	\]
In Example \ref{rp2twex} we use these classes to conclude as an $\M_2$-algebra
	\[
	H^{*,*}(Y;\underline{\Z/2})\cong \M_2[x,y]/(x^2= \tau y+\rho x, xy=0, y^2=0),~~~|x|=(1,1), |y|=(2,1).
	\]
\end{example}

\begin{rmk} The first two examples give a flavor of fundamental classes for nonfree submanifolds. It is natural to ask if such things can be defined for free submanifolds. We see this in the next example.
\end{rmk}

\begin{example}\label{introex3}
Let $Z$ denote the $C_2$-space whose underlying space is the genus two torus and whose $C_2$-action is given by a rotation action with two fixed points, as illustrated below.
The cohomology of this space is given by
	\[
	H^{*,*}(Z;\underline{\Z/2}) \cong \M_2\oplus \left(\Sigma^{1,0}A_0\right)^{\oplus 2} \oplus \Sigma^{2,2}\M_2.
	\]
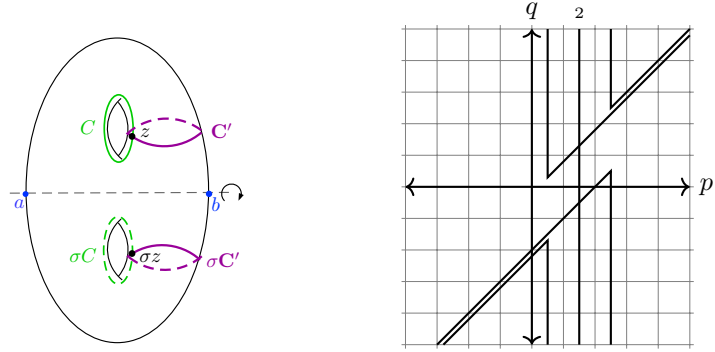
\begin{figure}[ht]
	\begin{subfigure}[b]{0.45\textwidth}
	\centering
	\includegraphics[scale=0.75]{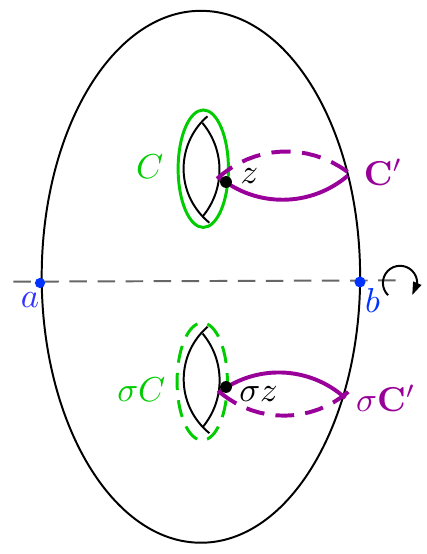}
	\end{subfigure}
	\begin{subfigure}[b]{0.45\textwidth}
	\centering
	\begin{tikzpicture}[scale=0.42]
		\draw[help lines,gray] (-4.125,-5.125) grid (5.125, 5.125);
		\draw[thick,<->] (-4,0)--(5,0)node[right]{$p$};
		\draw[thick,<->] (0,-5)--(0,5)node[above]{$q$};
		\draw[thick] (0.5,5)--(0.5,0.3)--(5,4.8);
		\draw[thick] (0.5,-5)--(0.5,-1.7)--(-2.8,-5);
		\draw[thick] (2.5,-5)--(2.5,.5)--(-3,-5);
		\draw[thick] (2.5,5)--(2.5,2.5)--(5,5);
		\draw[thick] (1.5,-5)--(1.5,5);
		\lab{1}{2}{black};
	\end{tikzpicture}
	\end{subfigure}
	\caption{The space $Z$ and the $\M_2$-module $H^{*,*}(Z;\underline{\Z/2})$.}
\end{figure}

In this example the cohomology is an infinitely generated, nonfree $\M_2$-module. The free submanifold $C\sqcup \sigma C$ is a codimension one submanifold, so we would expect to have a class $[C\sqcup \sigma C]\in H^{1,?}(Z;\underline{\Z/2})$. The submanifold is free, so there is no fixed set to determine the weight as in the previous examples. We will show there are actually infinitely many fundamental classes, one in each weight. That is, we have classes $[C\sqcup \sigma C]_q\in H^{1,q}(Z;\underline{\Z/2})$ for all integers $q$ that are related via the formula
	\[
	\tau\cdot [C\sqcup \sigma C]_q=[C\sqcup \sigma C]_{q+1}.
	\]
The submodule generated by these classes corresponds to a $\Sigma^{1,0}A_0$-summand. One can show the fundamental classes of the free submanifold $C'\sqcup \sigma C'$ generate the other summand. Lastly, there are classes $[a],[b]\in H^{2,2}(Z;\underline{\Z/2})$ corresponding to the two fixed points. Either class will generate a free summand in bidegree $(2,2)$, and the classes are related via $[a]=[b]+\rho^2\cdot 1$.

We can also choose a point $z\in Z\setminus Z^{C_2}$ and consider the classes $[z\sqcup \sigma z]_q\in H^{2,q}(Z;\underline{\Z/2})$. One might expect these classes to all be zero, but in fact, they are nonzero whenever $q\leq 0$. The intersection product then leads to some interesting relations as shown in Example \ref{tspit}.

We mention two other nonfree classes one might consider. There are two copies of a circle with a reflection action in the space $Z$. We can consider the circle that travels around the equator through $a$ and $b$, and the circle that travels around the perimeter of the picture through $a$ and $b$. Call these circles $D$ and $E$, respectively. Then $[D]=\rho\cdot 1$ while $[E]=[C\sqcup \sigma C]_{1}+\rho\cdot 1$.
\end{example}

\subsection{Nonequivariant fundamental classes}
Before trying to define equivariant fundamental classes, let's recall nonequivariant fundamental classes. Let $X$ be a closed manifold and $Y\subset X$ be a closed, connected submanifold of codimension $k$. We can define the fundamental class $[Y]\in H^{k}_{sing}(X;\Z/2)$ using the classical Thom isomorphism theorem from \cite{T}.

Consider the normal bundle $N$ of $Y$ in $X$. Let $U\subset X$ be a tubular neighborhood of $Y$. By excision
	\[
	H^{k}_{sing}(X,X-Y;\Z/2)\cong H^{k}_{sing}(U,U-Y;\Z/2) \cong H^{k}(N,N-0;\Z/2).
	\]
By the Thom isomorphism theorem, the righthand group is $\Z/2$ and generated by the Thom class $u_N$. Thus there exists a unique nonzero class in $H^{k}(X,X-Y;\Z/2)$. We now define the fundamental class $[Y]\in H^{k}(X;\Z/2)$ to be the image of this unique class under the induced map from the inclusion of the pairs $(X,\emptyset)\hookrightarrow (X,X-Y)$. Recall these classes have a nice intersection product. If $Y$ and $Z$ are two submanifolds of $X$ that intersect transversally, then $[Y]\smile [Z]=[Y\cap Z]$.

One goal of this paper is to define an equivariant analog to these classes in Bredon cohomology. Given an equivariant submanifold, the above hints that we should consider the normal bundle, and use some fact about the cohomology of the corresponding Thom space. Unfortunately, no direct analog of the Thom isomorphism theorem exists for general $C_2$-vector bundles in $\underline{\Z/2}$-coefficients; see Example \ref{mobiusovers11} for an example of a vector bundle $E$ such that the cohomology of $(E,E-0)$ is not just a shift of the cohomology base space. Despite this failure, we can still prove a weaker version of the Thom isomorphism theorem, and this is enough to define fundamental classes.

\subsection{The main theorems}
We now state the main results of this paper. In what follows, let $A_n$ denote the $\M_2$-module $A_n=H^{*,*}(S^n_a;\underline{\Z/2})\cong \tau^{-1}\M_2/(\rho^{n+1})$. It was shown in \cite[Theorem 5.1]{M2} that as a module over the cohomology of a point, the cohomology of any finite $C_2$-CW complex is isomorphic to a direct sum of free modules and shifted copies of $A_n$ for various values of $n$.

We begin with nonfree $C_2$-vector bundles. Let $X$ be a finite, nonfree $C_2$-CW complex and let $\pi:E\to X$ be a real $n$-dimensional $C_2$-vector bundle (precise definitions can be found in Section \ref{ch:eqvb}). Let $X_1,\dots,X_r$ denote the connected components of the fixed set $X^{C_2}$. As explained in Section \ref{ch:eqvb}, there exist weights $q_1$,\dots, $q_r$ such that for all $x\in X_i$, the fiber $E_x\cong \R^{n,q_i}$. Let $q$ be the maximum such weight. We will prove the following where all coefficients are understood to be $\underline{\Z/2}$. 

\begin{theorem}\label{intro1}
Let $X$, $X_i$, $\pi:E\to X$, $n$, $q_i$ and $q$ be defined as above and let $E'=E-0$. There exists a unique class $u_E\in H^{n,q}(E,E')$ that restricts to the nonzero class over both free and nonfree fibers. Furthermore this class satisfies the following properties:
\begin{enumerate}
	\item[(i)] $\psi(u_E)$ is the singular Thom class, where $\psi:H^{n,k}(E,E')\to H^n_{sing}(E,E')$ is the forgetful map;
	\item[(ii)] $\M_2\cdot u_E \cong \Sigma^{n,q}\M_2$, where $\M_2\cdot u_E$ denotes the submodule generated by $u_E$;
	\item[(iii)] If $Y$ is nonfree, then for $f:Y\to X$, we have the naturality formula $f^*u_{E}= \tau^ju_{f^*E}$ for a determined value of $j$;
	\item[(iv)] Under the inclusion of fixed points $j:X^{C_2}\to X$, we have that $j^*(u_E) = \sum_{i=1}^m \tau^{q-q_i}\rho^{q_i}u_{E_i^{C_2}}$ where $E_i = E|_{X_{i}}$;
	\item[(v)] The map $\phi_E=\pi^*(-) \smile u_E:H^{f,g}(X)\to H^{f+n,g+q}(E,E')$ is an isomorphism if $g\geq f$;
	\item[(vi)] Suppose $H^{*,*}(X) \cong (\oplus_{i=1}^c \Sigma^{k_i,\ell_i}\M_2) \oplus (\oplus_{j=1}^d \Sigma^{s_j,0}A_{r_j}).$ Then $H^{*,*}(E,E') \cong (\oplus_{i=1}^c \Sigma^{k_i+n,\ell'_i}\M_2) \oplus (\oplus_{j=1}^d \Sigma^{s_j+n,0}A_{r_j})$
	 where the weights $\ell_i'$ satisfy $\ell_i+q\geq \ell_i'\geq 0$;
	\item[(vii)] If in fact $E_x\cong E_y$ for all $x,y\in X^{C_2}$, then $\phi_E$ is an isomorphism in all bidegrees and $H^{*,*}(X)\cong H^{*+n,*+q}(E,E')$
\end{enumerate} 
\end{theorem}

Let $Y$ be a nonfree $k$-codimensional equivariant submanifold of $X$ and let $N$ denote the equivariant normal bundle. We can construct an equivariant tubular neighborhood of $Y$, and then use the class $u_{N}$ to get a class $[Y]\in H^{k,q}(X;\underline{\Z/2})$. As with the nonequivariant classes, we have a nice formula for how these classes multiply. For now, we state a summary. Recall $\tau\in H^{0,1}(pt;\underline{\Z/2})$.

\begin{theorem}\label{introprod}
Let $X$ be a nonfree, $n$-dimensional $C_2$-manifold, and let $Y$ and $Z$ be two closed, nonfree equivariant submanifolds. Suppose $Y$ intersects $Z$ transversally in the nonequivariant sense and that $Y\cap Z$ is a nonfree submanifold. Then there is a unique integer $j\geq 0$ such that
	\[
	[Y]\smile[Z] = \tau^{j}[Y\cap Z].
	\]
\end{theorem}
The exact value of $j$ is dependent on $Y$, $Z$, and $Y\cap Z$ and is given explicitly in Theorem \ref{prod}. \medskip

For free $C_2$-vector bundles, we will prove something similar. Again the coefficients are understood to be $\underline{\Z/2}$.

\begin{theorem}\label{intro2}
Let $X$ be a free, finite $C_2$-CW complex and let $\pi: E\to X$ be a real $C_2$-vector bundle and $E'=E-0$. For every integer $q$ there exists a unique class $u_{E,q}\in H^{n,q}(E,E')$ that restricts to the nonzero class over fibers. Furthermore this class satisfies the following properties:
\begin{enumerate}
	\item[(i)] $\psi(u_{E,q})$ is the singular Thom class, where $\psi:H^{n,q}(E,E')\to H^n_{sing}(E,E')$ is the forgetful map;
	\item[(ii)] $\tau\cdot u_{E,q} = u_{E,q+1}$;
	\item[(iii)] If $Y$ is free, then for $f:Y\to X$, we have the naturality formula $f^*u_{E,q}=u_{f^*E,q}$;
	\item[(iv)] The map $\pi^{*}(-)\smile u_{E,q}: H^{*,*}(X) \to H^{*+n,*+q}(E,E')$ is an isomorphism for all $q$. In particular, $H^{*,*}(X)\cong H^{*+n,*}(E,E')$.
\end{enumerate}
\end{theorem}

Now given a free submanifold, we can use the classes corresponding to the normal bundle to define fundamental classes $[Y]_q\in H^{k,q}(X;\underline{\Z/2})$ for all $q$. There is also a nice intersection product for these free fundamental classes. 

\begin{theorem}
Let $X$ be an $n$-dimensional $C_2$-manifold, and suppose $Y$ and $Z$ are equivariant submanifolds that intersect transversally in the nonequivariant sense and whose intersection is free. We have the following cases for the product of their fundamental classes. 
\begin{itemize}
	\item Suppose $Y$ and $Z$ are nonfree and their fundamental classes have weights $q$, $r$, respectively. Then
	\[
	[Y]\smile[Z] = [Y\cap Z]_{q+r}.
	\]
	\item Suppose $Y$ is nonfree and $Z$ is free. Then for every $r$,
	\[
	[Y]\smile[Z]_r = [Y\cap Z]_{q+r}.
	\]
	\item Suppose $Y$ and $Z$ are both free. Then for every $r$, $s$,
	\[
	[Y]_r\smile[Z]_s = [Y\cap Z]_{r+s}.
	\]
\end{itemize}
\end{theorem}

In the last section of this paper, we show these classes give a geometric interpretation for the Bredon cohomology of $C_2$-surfaces. Specifically, we prove the following theorem.

\begin{theorem}\label{intro3}
As an $\M_2$-module, the Bredon cohomology of any $C_2$-surface is generated by fundamental classes of submanifolds. 
\end{theorem}

\begin{rmk} 
There are many examples of $C_2$-surfaces whose Bredon cohomology is infinitely generated as an $\M_2$-module; see Example \ref{introex3}. This is unsurprising given that the modules $A_n=\tau^{-1}\M_2/(\rho^{n+1})$ are infinitely generated. Though, there is always a finite list of submanifolds whose fundamental classes generate the cohomology, and any free $A_i$-summand is generated by fundamental classes of free submanifolds. 
\end{rmk}

\subsection{Organization of the paper} 
There are two main topics in this paper, fundamental classes and the cohomology of Thom spaces. We define fundamental classes using Thom classes of normal bundles, so the topics are certainly related, but the full discussion of one is not required for the other. Section \ref{ch:backcoh} reviews relevant information about Bredon cohomology that is needed for either topic. After that, Section \ref{ch:eqvb} discusses Thom classes and the cohomology of Thom spaces. Section \ref{ch:funclasses} defines fundamental classes using the main theorems in Section \ref{ch:eqvb}. In Section \ref{ch:backsurf} we provide the necessary background on $C_2$-surfaces in order to then prove Theorem \ref{intro3} in Section \ref{ch:funclassessurf}. Finally Appendix \ref{ch:algproof} contains a technical algebraic proof needed to prove property (vi) of Theorem \ref{intro1}.

\subsection{Acknowledgements}
The work in this paper was part of the author's thesis project at University of Oregon. The author would like to thank her doctoral advisor Dan Dugger for all of his guidance and for many helpful conversations. The author would also like to thank Robert Lipshitz for his support. Lastly many thanks to the anonymous referee for their careful reading of this paper and for their helpful advice. 

\section{Background on Bredon cohomology}\label{ch:backcoh}
In this section, we review some preliminary facts about $RO(G)$-graded Bredon cohomology when $G=C_2$. While we regard this theory as graded on $RO(C_2)$, it should be noted Bredon cohomology is not canonically graded on $RO(G)$, and in general, one should be careful when working with these theories. These grading issues do not arise in this paper, so we do not expand on that point here; instead see \cite[Chapters IX.5 and XIII.1]{M2} for a more careful discussion of the grading.

\subsection{Mackey functors} The coefficients of any $RO(G)$-graded Bredon cohomology theory are what is known as a Mackey functor. In general, the definition of a Mackey functor requires some work, and a general exposition of Mackey functors can be found in Shulman's thesis \cite[Section 2.2]{Sh} and in May's book \cite[Chapter IX.4]{M2}. In the case of $G=C_2$, the definition can be distilled to the following.

\begin{definition} A \textbf{Mackey functor} $M$ for $G=C_2$ is the data of 
\begin{center}
\begin{tikzcd}[column sep=tiny]
M:&M(C_2)\arrow[out=-120, in=-60, loop, distance=0.5cm, "t^*" below ] \arrow[rr,shift right=.6ex,"p_*" below] & ~&M(\ast)\arrow[ll,shift right=.6ex,"p^*" above]
\end{tikzcd}
\end{center}
where $M(C_2)$ and $M(\ast)$ are abelian groups, and $p^*$, $p_*$, $t^*$ are homomorphisms that satisfy 
\begin{enumerate}
\item[(i)] $(t^*)^2 = id$,
\item[(ii)] $t^*\circ p^*=p^*$,
\item[(iii)] $p_*\circ t^*=p_*$, and
\item[(iv)] $p^*\circ p_*=1+t^*$.
\end{enumerate}
\end{definition}

Given an abelian group $B$, we can form the \emph{constant Mackey functor} $\underline{B}$ where $\underline{B}(C_2)=\underline{B}(\ast)=B$, $t^*=id$, $p_*=2$, and $p^*=id$. We will be concerned with the following constant Mackey functor:
\begin{center}
	\begin{tikzcd}[column sep=tiny]
	\underline{\Z/2}:& \Z/2\arrow[out=-120, in=-60, loop, distance=0.5cm, "1" below ] 	\arrow[rr,shift right=.6ex,"0" below] & ~& \Z/2\arrow[ll,shift right=.6ex,"1" above]
	\end{tikzcd}
\end{center}

\subsection{Bigraded theory} 
For a group $G$, Bredon cohomology is graded on $RO(G)$, the Grothendieck ring of finite-dimensional, real, orthogonal $G$-representations. When $G$ is the cyclic group of order two, any such $C_2$-representation $V$ is isomorphic to a direct sum of copies of the trivial representation $\R_{triv}$ and copies of the sign representation $\R_{sgn}$. Up to isomorphism, $V$ is entirely determined by its dimension and the number of sign representations appearing in this decomposition. It follows that $RO(C_2)$ is a rank $2$ free abelian group with generators given by $[\R_{triv}]$ and $[\R_{sgn}]$. For brevity, we will write $\R^{p,q}$ for the representation $\R_{triv}^{p-q}\oplus \R_{sgn}^{q}$. We will also write $\R^{p,q}$ for the element of $RO(C_2)$ that is equal to $(p-q)[\R_{triv}] + q[\R_{sqn}]$. For the cohomology groups, we will write $H^{p,q}(X;M)$ for the cohomology group $H^{\R^{p,q}}(X;M)$.

Given any finite-dimensional, real, orthogonal, $G$-representation $V$ we can form the one-point compactification $\hat{V}$. This new space will be an equivariant sphere which we will denote $S^V$; such spaces are referred to as \textbf{representation spheres}. Using these representation spheres, we can form equivariant suspensions. Whenever we have a based $G$-space $X$, we can form the $V$-th suspension of $X$ by
	\[
	\Sigma^VX=S^V \wedge X.
	\]
Note the basepoint must be a fixed point. Often when working with free spaces we will add a disjoint basepoint in order to form suspensions and cofiber sequences. We use the common notation of $X_+$ for $X \sqcup \{*\}$ where the disjoint basepoint is understood to be fixed by the action. 

An important feature of Bredon cohomology is that we have suspension isomorphisms: given any finite dimensional, real, orthogonal $G$-representation, there are natural isomorphisms
	\[
	\Sigma^{V}: \tilde{H}^{\alpha}(-;M) \to \tilde{H}^{\alpha+V}(\Sigma^V(-);M).
	\]
Given a cofiber sequence of based $G$-spaces
	\[
	A\overset{f}{\to} X \to C(f)
	\] 
we can form the Puppe sequence
	\[
	A\to X \to C(f) \to \Sigma^{\mathbf{1}}A \to \Sigma^{\mathbf{1}}C(f) \to \Sigma^{\mathbf{1}}X \to \dots 
	\]
where $\mathbf{1}$ is the one-dimensional trivial representation. From the suspension isomorphism this yields a long exact sequence
	\[
	\H^{V}(A)\leftarrow \H^{V}(X)\leftarrow \H^{V}(C(f)) \leftarrow \H^{V-\mathbf{1}}(A)\leftarrow \H^{V-\mathbf{1}}(X)\leftarrow \dots
	\]
for each representation $V\in RO(G)$.

When $G=C_2$, we have already discussed how the Bredon cohomology theory is a bigraded theory, and we will carry this notation over when discussing representation spheres and equivariant suspensions. In particular, we will denote $S^{\R^{p,q}}$ by $S^{p,q}$ and for a based space $X$ we will denote $\Sigma^{\R^{p,q}}X$ by $\Sigma^{p,q}X$. Translating the above into this notation, we have natural isomorphisms
	\[
	\Sigma^{p,q}:\tilde{H}^{a,b}(-;M) \to \tilde{H}^{a+p,b+q}(\Sigma^{p,q}(-);M)
	\]
for all $p,q\geq 0$. Given a cofiber sequence we have long exact sequences
	\[
	\dots\to\H^{p,q}(C(f))\to H^{p,q}(X)\to H^{p,q}(A) \to \H^{p+1,q}(C(f))\to H^{p+1,q}(X)\to \dots
	\]
for each $q\in \Z$.

Three particular representation spheres will appear often in this paper, namely $S^{1,1}$, $S^{2,1}$, and $S^{2,2}$. We include an illustration of these equivariant spheres in Figure \ref{fig:spheres}. The fixed set is shown in {\textbf{\color{blue}{blue}}} while the arrow is used to indicate the action of $C_2$ on the space. \smallskip

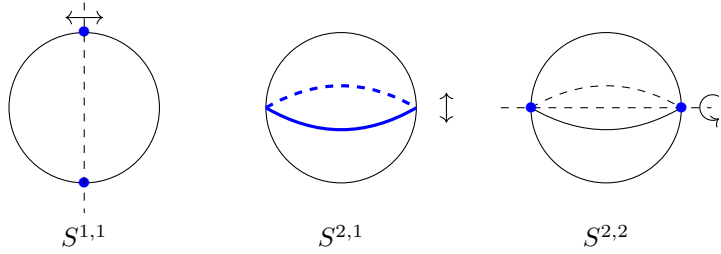
\begin{figure}[ht]
	\begin{tikzpicture}
		\draw (0,0) circle (1cm);
		\draw[blue] (0,-1) node{$\bullet$};
		\draw[blue] (0,1) node{$\bullet$};
		\draw[dashed] (0,1.4)--(0,-1.4);
		\draw[<->] (-.25,1.2)--(.25,1.2);
		\draw (0,-1.7) node{$S^{1,1}$};
	\end{tikzpicture}\hspace{.5in}
	\begin{tikzpicture}
		\draw (0,0) circle (1cm);
		\draw[blue,very thick] (-1,0) to[out=330,in=210](1,0);
		\draw[blue,very thick, dashed] (-1,0) to[out=30,in=150](1,0);
		\draw (0,-1.7) node{$S^{2,1}$};
		\draw[<->] (1.4,-.2)--(1.4,.2);
	\end{tikzpicture}\hspace{.2in}
	\begin{tikzpicture}
		\draw (0,0) circle (1cm);
		\draw (-1,0) to[out=330,in=210](1,0);
		\draw[dashed] (-1,0) to[out=30,in=150](1,0);
		\draw[dashed] (-1.4,0)--(1.4,0);
		\draw[->] (1.5,.15) arc (60:300:.17cm);
		\draw (0,-1.7) node{$S^{2,2}$};
		\draw[blue] (1,0)node{$\bullet$};
		\draw[blue] (-1,0)node{$\bullet$};
	\end{tikzpicture}
	\caption{Some representation spheres.}
	\label{fig:spheres}
\end{figure}

\subsection{The cohomology of orbits} 
Given any $C_2$-space $X$ we have an equivariant map $X\to pt$ where $pt$ denotes a single point with the trivial action. On cohomology, this gives a map of rings $H^{*,*}(pt;\underline{\Z/2}) \to H^{*,*}(X;\underline{\Z/2})$. Thus the cohomology of $X$ is a module over the cohomology of a point, which we denote $\M_2=H^{*,*}(pt;\underline{\Z/2})$.  Below we describe the cohomology of $pt=C_2/C_2$ as well as the cohomology of the free orbit $C_2$. These computations have been done many times, and are often attributed to unpublished notes of Stong.  The computation for coefficients in any constant Mackey functor can be found in work of Lewis \cite[Section 2]{L}. A computation for constant integer coefficients can also be found in work of Dugger \cite[Appendix B]{D1}, and the same methods used there can be used to compute the cohomology of orbits in constant $\Z/2$ coefficients.

In $\underline{\Z/2}$-coefficients, the cohomology of a point is given by the trivial square zero extension 
	\[
	\M_2\cong \Z/2[\rho,\tau] \oplus \frac{\Z/2(\rho,\tau)}{\Z/2[\rho,\tau]}\{\theta\}
	\]
where $|\rho|=(1,1)$, $|\tau|=(0,1)$, and $|\theta|=(0,-2)$. This ring is illustrated in the left-hand grid shown in Figure \ref{fig:point}. The $(p,q)$ spot on the grid refers to the $\R^{p,q}$-cohomology group. Each dot represents a copy of $\Z/2$, and we adopt the convention that the $(p,q)$ group is plotted up and to right of the $(p,q)$ coordinate. For example, $H^{0,0}(pt;\underline{\Z/2})$ is isomorphic to $\Z/2$, while $H^{1,0}(pt;\underline{\Z/2})$ is zero. 

\begin{figure}[ht]
	\begin{tikzpicture}[scale=.43]
		\draw[help lines,gray] (-5.125,-5.125) grid (5.125, 5.125);
		\draw[<->] (-5,0)--(5,0)node[right]{$p$};
		\draw[<->] (0,-5)--(0,5)node[above]{$q$};
		\foreach \y in {0,1,2,3,4}
			\draw (0.5,\y+.5) node{\small{$\bullet$}};
		\cone{0}{0}{black};
		\foreach \y in {1,2,3,4}
			\draw (1.5,\y+.5) node{\small{$\bullet$}};
		\foreach \y in {2,3,4}
			\draw (2.5,\y+.5) node{\small{$\bullet$}};
		\foreach \y in {3,4}
			\draw (3.5,\y+.5) node{\small{$\bullet$}};
		\foreach \y in {4}
			\draw (4.5,\y+.5) node{\small{$\bullet$}};
		\foreach \y in {-1,-2,-3,-4}
			\draw (.5,\y-.5) node{\small{$\bullet$}};
		\cone{0}{0}{black};
		\foreach \y in {-2,-3,-4}
			\draw (-.5,\y-.5) node{\small{$\bullet$}};
		\foreach \y in {-3,-4}
			\draw (-1.5,\y-.5) node{\small{$\bullet$}};
		\foreach \y in {-4}
			\draw (-2.5,\y-.5) node{\small{$\bullet$}};
		\cone{0}{0}{black};
		\draw[thick](1.5,5)--(1.5,1.5)node[below, right]{$\rho$};
		\draw[thick](2.5,2.5)--(2.5,5);
		\draw[thick](3.5,3.5)--(3.5,5);
		\draw[thick](4.5,4.5)--(4.5,5);
		\draw[thick](.5,1.5)node[xshift=-2.2ex]{$\tau$}--(4,5);
		\draw[thick](.5,2.5)--(3,5);
		\draw[thick](.5,3.5)--(2,5);
		\draw[thick](.5,4.5)--(1,5);
		\draw[thick](-.5,-2.5)--(-.5,-5);
		\draw[thick](-1.5,-3.5)--(-1.5,-5);
		\draw[thick](-2.5,-4.5)--(-2.5,-5);
		\draw[thick](.5,-2.5)--(-2,-5);
		\draw[thick](.5,-3.5)--(-1,-5);
		\draw[thick](.5,-4.5)--(0,-5);
		\draw (1,-1.5) node{$\theta$};
		\draw (1,-2.5) node{$\frac{\theta}{\tau}$};
		\draw (1,-3.5) node{$\frac{\theta}{\tau^2}$};
		\draw (1,-4.5) node{$\frac{\theta}{\tau^3}$};
		\draw (-.7,-1.7) node{$\frac{\theta}{\rho}$};
		\draw (-1.7,-2.7) node{$\frac{\theta}{\rho^2}$};
		\draw (-2.7,-3.7) node{$\frac{\theta}{\rho^3}$};
	\end{tikzpicture}
	\hspace{0.5in}
	\begin{tikzpicture}[scale=.43]
		\draw[help lines,gray] (-5.125,-5.125) grid (5.125, 5.125);
		\draw[<->] (-5,0)--(5,0)node[right]{$p$};
		\draw[<->] (0,-5)--(0,5)node[above]{$q$};
		\cone{0}{0}{black};
	\end{tikzpicture}
	\caption{ The ring $\M_2=H^{*,*}(pt; \underline{\Z/2})$.}
	\label{fig:point}
\end{figure}
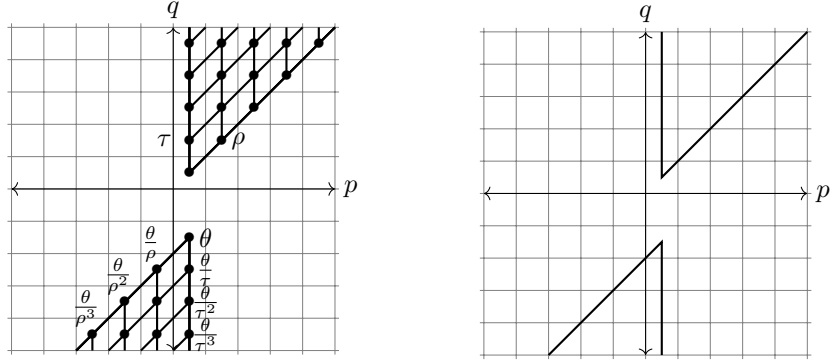

We will often refer to portion of the cohomology in the first quadrant as the \textbf{top cone} and refer to the other portion as the \textbf{bottom cone}. The top cone is polynomial in the elements $\rho$ and $\tau$. Multiplication by $\tau$ is indicated with vertical lines, and multiplication by $\rho$ is indicated with diagonal lines. The nonzero element $\theta$ in bidegree $(0,-2)$ is divisible by all non-zero elements in the top cone. 

In practice, it is often easier to work with an abbreviated picture, which is given on the right-hand grid in the above figure. It is understood that there is a $\Z/2$ at each spot within the top cone and within the bottom cone with the relations described above.

Note we now know the cohomology of all representation spheres by the suspension isomorphism. Specifically,
	\[
	\H^{*,*}(S^{p,q};\underline{\Z/2})\cong \Sigma^{p,q}\M_2
	\]
where $\Sigma^{p,q}\M_2$ denotes the free $\M_2$-module generated in bidegree $(p,q)$. Visually, we just shift the picture $(p,q)$ units.
 \medskip

We also include the cohomology of the free orbit $C_2$. As a ring, $H^{*,*}(C_2;\underline{\Z/2})$ is isomorphic to $\Z/2[u,u^{-1}]$ where $u$ is in bidegree $(0,1)$. As an $\M_2$-module, $H^{*,*}(C_2;\underline{\Z/2})$ is isomorphic to $\tau^{-1}\M_2/(\rho)$. See Figure \ref{fig:free} for the pictorial representation of this module and its abbreviated version. In these module pictures, action by $\tau$ is indicated by vertical lines, while action by $\rho$ is indicated by diagonal lines. 

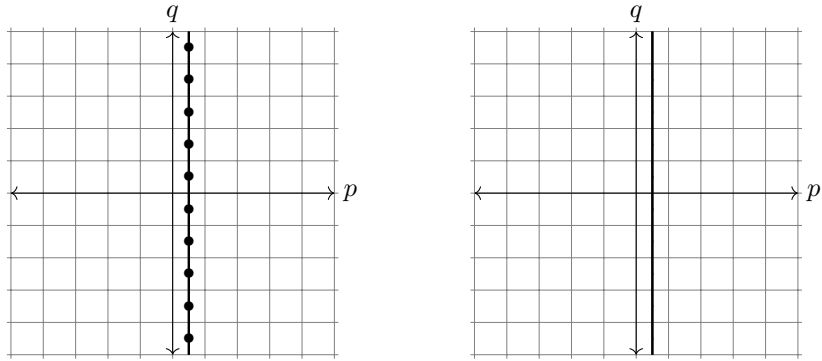
\begin{figure}[ht]
	\begin{tikzpicture}[scale=.43]
		\draw[help lines,gray] (-5.125,-5.125) grid (5.125, 5.125);
		\draw[<->] (-5,0)--(5,0)node[right]{$p$};
		\draw[<->] (0,-5)--(0,5)node[above]{$q$};
		\anti{0}{0}{black};
		\foreach \y in {-5,...,4}
			\draw (0.5,\y+.5) node{\small{$\bullet$}};
	\end{tikzpicture}\hspace{0.5 in}
	\begin{tikzpicture}[scale=.43]
		\draw[help lines,gray] (-5.125,-5.125) grid (5.125, 5.125);
		\draw[<->] (-5,0)--(5,0)node[right]{$p$};
		\draw[<->] (0,-5)--(0,5)node[above]{$q$};
		\anti{0}{0}{black};
	\end{tikzpicture}
\caption{The cohomology of $C_2$ as an $\M_2$-module.}
\label{fig:free}
\end{figure} 

\subsection{Cohomology of the antipodal spheres}
We review the cohomology of another family of spheres. The importance of these spaces will be clear shortly. Let $S^n_a$ denote the equivariant $n$-sphere whose $C_2$-action is given by the antipodal map. The cohomology of the antipodal spheres is given by 
	\[
	H^{*,*}(S^n_a;\underline{\Z/2}) \cong\tau^{-1}\M_2/(\rho^{n+1})
	\]
as an $\M_2$-module. This result is well known, but a detailed computation can be found in \cite[Example 3.4]{Ha}. The module $\tau^{-1}\M_2/(\rho^{n+1})$ can be represented pictorially as shown in Figure \ref{fig:anti}. 
\begin{figure}[ht]
	\begin{tikzpicture}[scale=.43]
		\draw[help lines,gray] (-1.125,-5.125) grid (7.125, 5.125);
		\draw[<->] (-1,0)--(7,0)node[right]{$p$};
		\draw[<->] (0,-5)--(0,5)node[above]{$q$};
		\foreach \x in {0,...,5}
		\draw[thick] (\x+1/2, -5)--(\x+1/2, 5);
		\draw[thick] (1/2,-4.5)--(5.5,.5);
		\draw[thick] (1/2,-3.5)--(5.5,1.5);
		\draw[thick] (1/2,-2.5)--(5.5,2.5);
		\draw[thick] (1/2,-1.5)--(5.5,3.5);
		\draw[thick] (1/2,-.5)--(5.5,4.5);
		\draw[thick] (1/2, .5)--(5,5);
		\draw[thick] (1/2, 1.5)--(4,5);
		\draw[thick] (1/2, 2.5)--(3,5);
		\draw[thick] (1/2, 3.5)--(2,5);
		\draw[thick] (1/2, 4.5)--(1,5);
		\draw[thick] (5.5, -.5)--(1,-5);
		\draw[thick] (5.5, -1.5)--(2,-5);
		\draw[thick] (5.5, -2.5)--(3,-5);
		\draw[thick] (5.5, -3.5)--(4,-5);
		\draw[thick] (5.5, -4.5)--(5,-5);
		\foreach \x in {0,...,5}
			{
			\foreach \y in {-5,-4,-3,-2,-1,0,1,2,3,4}
				\draw (\x+.5,\y+.5) node{\small{$\bullet$}};
			}
	\end{tikzpicture}\hspace{0.5in}
	\begin{tikzpicture}[scale=.43]
		\draw[help lines,gray] (-1.125,-5.125) grid (7.125, 5.125);
		\draw[<->] (-1,0)--(7,0)node[right]{$p$};
		\draw[<->] (0,-5)--(0,5)node[above]{$q$};
		\foreach \x in {0,5}
			\draw[thick] (\x+1/2, -5)--(\x+1/2, 5);
		\draw[thick] (1/2,-4.5)--(5.5,.5);
		\draw[thick] (1/2,-3.5)--(5.5,1.5);
		\draw[thick] (1/2,-2.5)--(5.5,2.5);
		\draw[thick] (1/2,-1.5)--(5.5,3.5);
		\draw[thick] (1/2,-.5)--(5.5,4.5);
		\draw[thick] (1/2, .5)--(5,5);
		\draw[thick] (1/2, 1.5)--(4,5);
		\draw[thick] (1/2, 2.5)--(3,5);
		\draw[thick] (1/2, 3.5)--(2,5);
		\draw[thick] (1/2, 4.5)--(1,5);
		\draw[thick] (5.5, -.5)--(1,-5);
		\draw[thick] (5.5, -1.5)--(2,-5);
		\draw[thick] (5.5, -2.5)--(3,-5);
		\draw[thick] (5.5, -3.5)--(4,-5);
		\draw[thick] (5.5, -4.5)--(5,-5);
	\end{tikzpicture}
	\caption{The $\M_2$-module $\tau^{-1}\M_2/(\rho^{n+1})$ when $n=5$.}
	\label{fig:anti}
\end{figure}
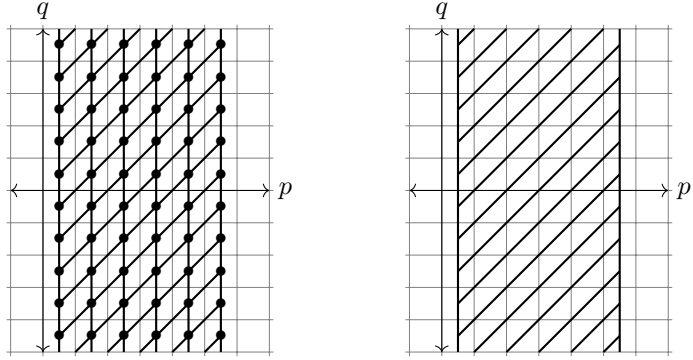

\subsection{Some helpful theorems} 
We conclude this section by collecting some helpful lemmas and theorems that will be used in the paper. Many of the statements involve finite $C_2$-CW complexes, so we first provide a definition of such spaces.

\begin{definition} 
A \textbf{$\mathbf{C_2}$-CW complex} is a $C_2$-space with a filtration $X^0\subset X^1\subset \dots$ such that $X^0$ consists of a disjoint union of orbits $C_2/C_2$ and $C_2$, and such that $X^{n+1}$ is obtained from $X^{n}$ by attaching cells of the form $C_2/H_\alpha\times D^n$ for $H_\alpha =C_2$ or $H_\alpha = \{e\}$ via the usual pushout diagram:
\begin{center}
	\begin{tikzcd}
	\bigsqcup_{\alpha} C_2/H_\alpha \times S^{n-1} \arrow[r] \arrow[d,hook]& X^{n-1}\arrow[d]\\
	\bigsqcup_{\alpha} C_2/H_\alpha \times D^{n} \arrow[r]& X^n
	\end{tikzcd}
\end{center}
\end{definition}

We begin with a lemma that relates the singular cohomology, the Bredon cohomology, and the action of $\rho$. This statement follows from considering the long exact sequence associated to $S^{0}\overset{\rho}{\to} S^{1,1} \to C_{2+}\wedge S^{1,1}$ and is originally due to \cite{AM}.
\begin{lemma}\label{forgetfulles}(The forgetful long exact sequence). 
Let $X$ be a pointed $C_2$-space. For every integer $q$, we have a long exact sequence
\begin{center}
	\begin{tikzcd}
		\arrow[r] & \tilde{H}^{p-1,q}(X) \arrow[r,"\rho \cdot"]& \tilde{H}^{p,q+1}(X) \arrow[r,"\psi"] & \tilde{H}^{p}_{sing}(X) \arrow[r] & \tilde{H}^{p,q}(X) \arrow[r] & ~
	\end{tikzcd}
\end{center}
where the coefficients are understood to be $\underline{\Z/2}$. 
\end{lemma}

We will refer to the map $\psi:\H^{p,q}(X) \to \H^{p}_{sing}(X)$ as the ``forgetful map". Note in $\M_2$, the element $\tau$ forgets to $1\in H^{0}_{sing}(pt)$, while $\rho$ forgets to zero. Indeed, by the exactness of the forgetful long exact sequence, for any $X$, a given cohomology class forgets to zero if and only if it is the image of $\rho$.

There are two types of spaces we will encounter in this paper whose cohomology is entirely dependent on the singular cohomology. These lemmas are likely well-known, but a recent proof in the same notation can be found in \cite[Section 3]{Ha}.

\begin{lemma}\label{trivial}
Let $X$ be a trivial finite $C_2$-CW complex. Then as $\M_2$-modules
	\[
	H^{*,*}(X;\underline{\Z/2})\cong \M_2\otimes_{\Z/2} H^{*}_{sing}(X;\Z/2).
	\]
\end{lemma}

\begin{lemma}\label{times} 
Let $Y$ be a finite $C_2$-CW complex. The cohomology as an $\M_2$-module of the free $C_2$-space $C_2\times Y$ is given by
	\[
	H^{*,*}(Y\times C_2; \underline{\Z/2})\cong \tau^{-1}\M_2/(\rho)\otimes_{\Z/2} H^{*}_{sing}(Y;\Z/2).
	\]
\end{lemma}
\begin{rmk}\label{freeiso}
In fact, the above can be generalized to state for any free, finite $C_2$-CW complex $Y$
	\[
	H^{*,*}(Y; \underline{\Z/2})\cong \Z/2[\rho,\tau^{\pm 1}]\underset{\Z/2[u]}{\otimes} H^{*}_{sing}(Y/C_2;\Z/2),
	\]
where the action of $u$ on the right is coming from Borel cohomology and the action on the left is coming from $u\mapsto \rho\tau^{-1}$. Borel cohomology is discussed in Section \ref{sec:borel} and the above isomorphism is explained in Remark \ref{freeboriso}.
\end{rmk}
We next mention an important theorem about the cohomology of $C_2$-manifolds. Here by ``$C_2$-manifold" we mean a piecewise linear manifold with a locally linear $C_2$-action. By closed, we simply mean a closed manifold in the nonequivariant sense. The proof of this theorem is given in \cite[Appendix A]{Ha}.

\begin{theorem}\label{topm2} 
Let $X$ be an $n$-dimensional, closed $C_2$-manifold with a nonfree $C_2$-action. Suppose $n-k$ is the largest dimension of submanifold appearing as a component of the fixed set. Then there is exactly one summand of $\H^{*,*}(X;\underline{\Z/2})$ of the form $\Sigma^{i,j}\M_2$ where $i\geq n$, and it occurs for $(i,j)=(n,k)$.
\end{theorem}

We conclude this section by recalling a structure theorem for the cohomology of finite $C_2$-CW complexes proven in \cite[Theorem 5.1]{M1}. Note it is important that $X$ is finite in the sense that it contains only finitely many cells. 

\begin{theorem}[\cite{M2}]\label{structure} For any finite $C_2$-CW complex $X$, we can decompose the $RO(C_2)$-graded cohomology of $X$ with constant $\underline{\Z/2}$-coefficients as
	\[
	H^{*,*}(X;\underline{\Z/2}) \cong (\oplus_i\Sigma^{p_i,q_i}\M_2 ) \oplus (\oplus_j \Sigma^{m_j ,0} A_{n_j} )
	\] 
as a module over $\M_2$ where $A_{n}$ denotes the cohomology of the $n$-sphere with the free antipodal action. Furthermore, the shifts are given by actual representations. That is, $p_i\geq q_i\geq 0$ for all $i$ and $m_j\geq 0$ for all $j$. 
\end{theorem}

When $X$ is free or has only one fixed point, we get a refined structure theorem that follows from the isomorphism discussed in Remark \ref{freeiso}.

\begin{corollary}\label{freestructure}
 If $X$ is a free finite $C_2$-CW complex, then 
	\[
	H^{*,*}(X;\underline{\Z/2}) \cong \oplus_j \Sigma^{m_j ,0} A_{n_j}.
	\]
Similarly if $X$ is a finite $C_2$-CW complex with exactly one fixed point, then
	\[
	\H^{*,*}(X;\underline{\Z/2}) \cong \oplus_j \Sigma^{m_j ,0} A_{n_j}.
	\]
\end{corollary}

Lastly we record a corollary of Theorem \ref{structure} that will be helpful later on.
\begin{lemma}\label{tauiso}
Let $X$ be a finite $C_2$-CW complex. For $k> 0$, action by $\tau^k$ gives an isomorphism $\tau^k:H^{f,g}(X)\to H^{f,g+k}(X)$ if $g\geq f$.  
\end{lemma}
\begin{proof} 
By inspection this holds for modules of the form $\Sigma^{s,0}A_r$ and $\Sigma^{p,q} \M_2$ where $p\geq q$. The statement then follows from Theorem \ref{structure}.
\end{proof}

\section{$C_2$-vector bundles and Thom isomorphism theorems}
\label{ch:eqvb}
In this section, we provide some background on $C_2$-vector bundles and then prove the Thom isomorphism theorems given as Theorem \ref{intro1} and Theorem \ref{intro2} in the introduction. Another approach to the Thom isomorphism is given in \cite{CW} and uses a grading system larger than the usual $RO(G)$-grading. Presumably there is a connection between the two approaches, but we haven't investigated this.

We begin by reviewing $C_2$-vector bundles. The following can be found in \cite[Section 1]{Se}.

\begin{definition} 
Let $X$ be a $C_2$-space. A \textbf{$\mathbf{C_2}$-vector bundle} over $X$ is the data of a nonequivariant vector bundle $\pi:E\to X$ such that $E$ is a $C_2$-space. Furthermore, $C_2$ should act on $E$ via vector bundle maps over the action of $C_2$ on $X$. Explicitly, the following diagram should commute where $\sigma$ denotes the action of $C_2$,
\begin{center}
	\begin{tikzcd} 
	E \arrow[r,"\sigma"] \arrow[d,"\pi"]&E\arrow[d,"\pi"]\\
	X\arrow[r,"\sigma"]&X
	\end{tikzcd}
\end{center}
and for each $x\in X$ the restriction of the action to the fibers
\begin{center}
	\begin{tikzcd} E_x\arrow[r,"\sigma"]&E_{\sigma x}
	\end{tikzcd}
\end{center}
should be a linear map.
\end{definition}

Many of the constructions for nonequivariant bundles exist for $C_2$-vector bundles. For example, if the base space is paracompact and Hausdorff, then any $C_2$-vector bundle can be given a $C_2$-invariant Euclidean metric that allows us to define the unit disk bundle and the unit sphere bundle, which in turn allows us to define the Thom space. Also, given an equivariant map $f:Y\to X$ and a vector bundle $E\to X$ we can form the pullback bundle $f^*E\to Y$. When working with these pullback bundles, the following important fact is still true; see \cite[Proposition 1.3]{Se} for a proof.

\begin{lemma}\label{homotopyinv} 
Let $X$ and $Y$ be $C_2$-CW complexes. Suppose $f,g:Y\to X$ are equivariantly homotopic and $E\to X$ is a $C_2$-vector bundle. Then $f^*E\cong g^*E$ as $C_2$-vector bundles over $Y$. 
\end{lemma}

The following lemma follows immediately from the above.

\begin{lemma}\label{constantrep} 
Let $E\to X$ be a finite dimensional $C_2$-vector bundle over a $C_2$-CW complex. Suppose $x,y$ are two fixed points contained in the same connected component of $X^{C_2}$. Then the fibers $E_x$ and $E_{y}$ are isomorphic as $C_2$-representations. 
\end{lemma} 

While many of the constructions and basic lemmas carry over from nonequivariant vector bundle theory, issues arise when we start considering cohomology. In particular, there is no direct analog of the Thom isomorphism theorem in Bredon cohomology that holds for general vector bundles in $\underline{\Z/2}$-coefficients, as seen in the following example. 
\begin{example}\label{mobiusovers11} 
Let $E\to S^{1,1}$ be the nontrivial one-dimensional bundle over $S^{1,1}$. An illustration of the disk bundle is shown below. As usual, the fixed set is shown in {\textbf{\color{blue}{blue}}}, while conjugate points are indicated by matching symbols.
\begin{figure}[ht]
	\includegraphics[scale=0.9]{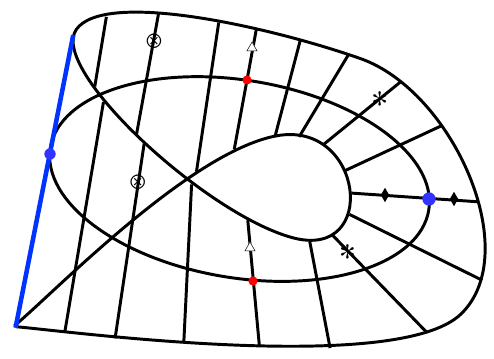}
	\caption{The M\"obuis bundle over $S^{1,1}$.}
\end{figure}

In this example, there are two components of the fixed set of the base space, both of which are isolated points. Over one point, the fiber is isomorphic to the $C_2$-representation $\R^{1,0}$; over the other point, the fiber is isomorphic to the $C_2$-representation $\R^{1,1}$. We have 
	\[
	H^{*,*}(E,E-0)\cong H^{*,*}(DE,SE)\cong \H^{*,*}(DE/SE)
	\]
where $DE$ and $SE$ are the unit disk and unit sphere bundle, respectively. Now $DE/SE$ is a familiar space: the underlying space is the projective plane, so in particular, it is a $C_2$-surface with exactly one fixed circle and one fixed point. The cohomology is given by Theorem \ref{nonfreeanswer} to be
	\[
	\H^{*,*}(DE/SE)\cong \Sigma^{1,1}\M_2\oplus \Sigma^{2,1}\M_2.
	\]
On the other hand, the cohomology of the base space is given by the suspension isomorphism to be 
	\[
	H^{*,*}(S^{1,1})\cong \M_2 \oplus \Sigma^{1,1}\M_2.
	\]
We see the cohomology of the Thom space is not a shift of the cohomology of the base space, but there are still similarities. There are the same number of free summands, and both summands are shifted by one topological dimension. There is also a unique class in $H^{*,*}(E,E-0)$ that generates a free summand and has topological dimension equal to the dimension of the bundle. Note the weight of this class corresponds to the maximum weight representation over the fixed set. This class will be the Thom class described in Theorem \ref{thomnonfree}. 
\end{example}

\subsection{The Thom isomorphism in bigraded Borel cohomology}\label{sec:borel} Before considering Bredon cohomology, we first analyze the Borel equivariant cohomology of $C_2$-equivariant Thom spaces. This analysis will prove useful in our proof of the Thom isomorphism in Bredon cohomology.

\begin{notation} In the remainder of this section, all coefficients will be understood to be $\underline{\Z/2}$. Given a vector bundle $E$ over a space $X$, we will write $E'$ for the complement of the zero section and $X^E$ for the Thom space.
\end{notation}

We begin by recalling the definition of Borel cohomology. As usual $EC_2$ denotes a contractible $C_2$-space with a free $C_2$-action. For example, one can take $S^{\infty}$ with the antipodal action.

\begin{definition}
Let $X$ be a $C_2$-space. The {\bf{bigraded Borel cohomology}} of $X$ is given by $H^{p,q}_{Bor}(X):=H^{p,q}(EC_2\times X)$. The {\bf{reduced bigraded Borel cohomology}} of a pointed $C_2$-space is given by $\H^{p,q}_{Bor}(X):=\H^{p,q}(EC_{2+}\wedge X)$
\end{definition}

The Borel cohomology of the orbits is given in the lemma below. The computation is straight-forward from the definition, so we leave the proof to the reader. 

\begin{lemma} As a ring $H^{*,*}_{Bor}(pt)\cong \Z/2[\rho,\tau^{\pm 1}]$ where $|\rho|=(1,1)$ and $|\tau|=(0,1)$. As a $\Z/2[\rho, \tau^{\pm 1}]$-algebra, $H^{*,*}_{Bor}(C_2)\cong \Z/2[\tau^{\pm 1}]$ with $\rho$ acting as zero.
\end{lemma}

\begin{rmk}
The quotient lemma implies \[H^{p,0}(EC_2\times X)\cong H^{p}_{sing}((EC_2\times X)/C_2)= H^{p}_{Bor}(X),\] so the bigraded Borel cohomology is indeed an extension of the usual Borel cohomology when working with constant coefficients. In fact, the bigraded Borel cohomology is entirely determined by the singly-graded Borel cohomology, as explained in the lemma below.
\end{rmk}
\begin{lemma}\label{boriso} Let $X$ be a finite $C_2$-CW complex and let $u$ denote the polynomial generator of $H^{*}_{Bor}(pt)\cong H^{*}_{sing}(\R P^\infty)\cong \Z/2[u]$. Then
	\[
	H^{*,*}_{Bor}(X)\cong H^{*}_{Bor}(X) \underset{\Z/2[u]}{\otimes} \Z/2[\rho, \tau^{\pm 1}]
	\]
where the module structure on the right factor is given by $u\mapsto \tau^{-1}\rho$.
\end{lemma}
\begin{proof}
Observe $\Z/2[\rho, \tau^{\pm 1}]\cong \Z/2[u][\tau^{\pm 1}]$ is a flat $\Z/2[u]$-module, so \[H^{*}_{Bor}(-) \underset{\Z/2[u]}{\otimes} \Z/2[\rho, \tau^{\pm 1}]\] is a bigraded cohomology theory for finite complexes. Now $H^{*}_{Bor}(X)=H^{*,0}_{Bor}(X)$ and we have a natural map
	\[
	H^{*}_{Bor}(X) \underset{\Z/2[u]}{\otimes} \Z/2[\rho, \tau^{\pm 1}] \to H^{*,*}_{Bor}(X)
	\]
given by $\alpha\otimes \rho^j\tau^k \mapsto \rho^j\tau^k\alpha$. One can check this gives an isomorphism when $X=pt$ and $X=C_2$, and thus the above is an isomorphism whenever $X$ is a finite $C_2$-CW complex. 
\end{proof}
\begin{rmk}\label{freeboriso} If $X$ is free, then the projection map $EC_2\times X \to X$ is a weak equivalence of $C_2$-spaces. Hence the Bredon cohomology and the Borel cohomology are isomorphic. Since $H^{*}_{Bor}(X)\cong H^{*}_{sing}(X/C_2)$ in the free case, this gives the isomorphism mentioned in Remark \ref{freeiso}.
\end{rmk}

We now prove there is a Thom isomorphism theorem for Borel cohomology. Let $\pi:E\to X$ be an $n$-dimensional $C_2$-vector bundle. In the theorem below $E_x$ denotes the fiber $\pi^{-1}(x)$ for a fixed point $x$, and $E_{y,\sigma y}$ denotes the fiber $\pi^{-1}(\{y,\sigma y\})$ for conjugate points $y, \sigma y$. Note $E_x\cong \R^{n,q}$ for some weight $q$ while $E_{y,\sigma y} \cong C_2\times \R^n$. On cohomology
	\[
	H^{*,*}_{Bor}(E_x,E_x')\cong \H^{*,*}_{Bor}(S^{n,q})\cong \Sigma^{n,q}\Z/2[\rho,\tau^{\pm 1}], \quad \text{ while} 
	\]
	\[
	H^{*,*}_{Bor}(E_{y,\sigma y},E_{y,\sigma y}') \cong \H^{*,*}_{Bor}(C_{2+}\wedge S^{n})\cong \Sigma^{n,0}\Z/2[\tau^{\pm 1}].
	\]
\begin{theorem}\label{borthom} Let $X$ be a finite $C_2$-CW complex and let $\pi:E\to X$ be an $n$-dimensional $C_2$-vector bundle. Then there is a unique class $u_E^{Bor}\in H^{n,0}_{Bor}(E,E')$ that satisfies the following properties:
\begin{enumerate}
\item[(i)] The class $u_E$ restricts to the unique nonzero class over both free and nonfree fibers. That is, if $x\in X^{C_2}$ then under the inclusion $i:E_x\to E$ the class $i^*(u_E)$ is the nonzero element in $H^{n,0}_{Bor}(E_x,E'_x)$. Similarly if $y\in X-X^{C_2}$ then under the inclusion $j:E_{y,\sigma y}\to E$, the class $j^*(u_E)$ is the unique nonzero class in $H^{n,0}_{Bor}(E_{y,\sigma y},E_{y,\sigma y}')$;
\item[(ii)] The map $\pi^{*}(-) \smile u_E^{Bor}: \H^{*,*}_{Bor}(X_+)\to \H^{*+n,*}_{Bor}(X^E)$ is an isomorphism. 
 \end{enumerate}
\end{theorem}
\begin{proof}
The singly graded Borel cohomology is defined using singular cohomology, so our strategy is to apply the singular cohomology Thom isomorphism theorem to get the corresponding theorem for Borel cohomology. We do this by making the following observation: if $Y$ is a free $C_2$-space and $E\to Y$ is a $C_2$-vector bundle, then $E/C_2 \to Y/C_2$ is also a vector bundle (this is discussed in the proof of \cite[Proposition 2.1]{Se}, for example). 

Consider the bundle $id\times \pi: EC_2\times E \to EC_2 \times X$. Applying the above observation, we have a vector bundle $id\times_{C_2} \pi: EC_2\times_{C_2} E \to EC_2\times_{C_2} X$. The usual Thom isomorphism theorem gives a class
	\[
	u\in \H^{n}_{sing}((EC_2\times_{C_2} X)^{EC_2\times_{C_2} E})
	\]
and the map
	\[
	(id\times \pi)^*(-)\smile u : \H^{j}_{sing}((EC_2\times_{C_2} X)_+) \to  \H^{n+j}_{sing}((EC_2\times_{C_2} X)^{EC_2\times_{C_2} E})
	\]
is an isomorphism. The domain is exactly $H^{j}_{Bor}(X)$ and for the codomain, there is a homeomorphism of spaces
	\[
	\phi:EC_{2+}\wedge_{C_2} X^E\to (EC_2\times_{C_2} X)^{EC_2\times_{C_2} E}
	\]
given by $y \wedge \alpha \mapsto (y,\alpha)$.
Let $u^{Bor}_E=\phi^*u\in \H^n_{sing}(EC_{2+}\wedge_{C_2} X^E)\cong \H^{n,0}_{Bor}(X^E)$. We have the following isomorphisms:
\begin{center}
\begin{tikzcd}[row sep=0.23cm]
 \H^{j}_{sing}(EC_{2+}\wedge_{C_2} X_+)\arrow[rrr, "(id\times \pi)^*(-)\smile u^{Bor}_E"]\arrow[d, "\rotatebox{90}{\(\sim\)}", equal]& & & \H^{j+n}_{sing}(EC_{2+}\wedge_{C_2} X^E)\arrow[d, "\rotatebox{90}{\(\sim\)}", equal]\\
  \H^{j,0}_{Bor}(X_+)\arrow[rrr, "\pi^*(-)\smile u^{Bor}_E"]&&& \H^{j+n,0}_{Bor}(X^E)
\end{tikzcd}
\end{center}
Since $\tau$ acts invertibly on the Borel cohomology, it follows that the bottom isomorphism holds for any weight, not just zero. Lastly the statements about restricting to fibers follow because these statements hold for $u$.
\end{proof}

\subsection{The Thom isomorphism in Bredon cohomology for trivial spaces} We next consider $C_2$-vector bundles where the base space has a trivial $C_2$-action. In this case, there is a Thom isomorphism theorem for Bredon cohomology that can be proven using the same methods as the singular Thom isomorphism theorem.

\begin{lemma}\label{trivthom}
Suppose $X$ is a finite, connected $C_2$-CW complex with a trivial $C_2$-action and $\pi:E\to X$ is a $C_2$-vector bundle. Suppose the fiber over some point $x\in X$ is isomorphic to $\R^{n,q}$. Then there is a unique class $u_E\in H^{n,q}(E,E')$ such that $u_E$ restricts to the nonzero class over fibers. Furthermore, the class $u_E$ satisfies the following properties:
\begin{enumerate}[label=(\roman*)]
\item The map $\phi_E:=\pi^*(-) \smile u_E : H^{*,*}(X) \to H^{*+n, *+q}(E,E')$ is an isomorphism; and
\item If $Y$ is another trivial $C_2$-CW complex and there is a map $f:Y\to X$, then $f^*u_E=u_{f^*E}$.
\end{enumerate} 
\end{lemma}
\begin{proof}
The proof of this theorem is analogous to the proof of the usual Thom isomorphism for finite CW complexes, so we only provide an outline. Using compactness, we can cover $X$ with finitely many neighborhoods $U_1,\dots, U_k$ such that $E|_{U_i} \cong \R^{n,q}\times U_i$. For trivial bundles, the existence of a Thom class $u_{i} \in H^{*,*}(E|_{U_i} , E|_{U_i}')$ follows from the suspension isomorphism. We can now paste these classes together using the Mayer-Vietoris sequence, just as in the proof of the nonequivariant Thom isomorphism theorem. Property (i) then follows from the five-lemma, and property (ii) follows by restricting to fibers and using the uniqueness of the Thom class. 
\end{proof}

\subsection{The Thom isomorphism theorem for Bredon cohomology} We are now ready to prove the main theorems of this section. Our goal is to prove there exist Thom classes in Bredon cohomology for $C_2$-vector bundles over finite $C_2$-CW complexes and that these classes satisfy a list of properties. The two theorems are stated separately for nonfree versus free base spaces and are each broken into two parts. We begin with some setup for the nonfree case.

Let $X$ be a finite, nonfree $C_2$-CW complex and $E\to X$ be an $n$-dimensional $C_2$-vector bundle. Let $X_1, \dots, X_m$ denote the connected components of the fixed set $X^{C_2}$. By Lemma \ref{constantrep}, for each $X_i$ there is an integer $q_i$ such that for all $x\in X_i$, $E_x\cong \R^{n,q_i}$. Let $q=\max\{q_1,\dots q_m\}$. We now state the first part of the theorems.

\begin{theorem}[Thom isomorphism for nonfree base spaces, Part I]\label{thomnonfree} Let $X$, $\pi:E\to X$, $n$, $X_i$, $q_i$ and $q$ be defined as above and let $E'=E-0$. There exists a unique class $u_E\in H^{n,q}(E,E')$ such that the following holds:
\begin{enumerate}[label=(\roman*)]
	\item For every $i$ and $x\in X_i$, the class $u_E$ restricts to $\tau^{q-q_i}\alpha_x$ where $\alpha_x$ is the generator of $H^{*,*}(E_x,E_x')\cong \H^{*,*}(S^{n,q_i})\cong \Sigma^{n,q_i}\M_2$;
	\item For every $x\in X\setminus X^{C_2}$, the class $u_E$ restricts to  the unique nonzero class in bidegree (n,q) in $H^{*,*}(E_{x,\sigma x},E_{x,\sigma x}')\cong \H^{*,*}(S^{n,0}\wedge C_{2+})\cong \Sigma^{n,0}A_0$ where $E_{x,\sigma x}=\pi^{-1}(\{x,\sigma\})\cong C_2\times \R^n$;
	\item The map $\phi_E:=\pi^*(-)\smile u_E: H^{f,g}(X)\to H^{f+n,g+q}(E,E')$ is an isomorphism in bidegrees $(f,g)$ whenever $g\geq f$; and
	\item If in fact $E_x\cong E_y$ for all $x,y\in X^{C_2}$, then $\phi_E$ is an isomorphism in all bidegrees and $H^{*,*}(X)\cong H^{*+n,*+q}(E,E')$. 
\end{enumerate} 
\end{theorem}
\begin{theorem}[Thom isomorphism for free base spaces, Part I]\label{thomfree}
Let $X$ be a free $C_2$-space and $E\to X$ be an $n$-dimensional $C_2$-vector bundle. Then for each $q\in \Z$ there is a unique class $u_{E,q}\in H^{n,q}(E,E')$ such that the following holds:
\begin{enumerate}[label=(\roman*)]
	\item For every pair of conjugate points $x,\sigma x\in X$, the class $u_{E,q}$ restricts to the unique nonzero element in bidegree $(n,q)$ in $H^{*,*}(E_{x,\sigma x},E_{x,\sigma x}')\cong\H^{*,*}(S^{n}\wedge C_{2+})\cong \Sigma^{n,0}A_0$;
	\item $\tau\cdot u_{E,q} = u_{E,q+1}$; and
	\item The map $\phi_{E,q}=\pi^{*}(-)\smile u_{E,q}$ is an isomorphism for all $q$;
\end{enumerate}
\end{theorem}
\begin{proof}[Proof of Theorem \ref{thomnonfree} and Theorem \ref{thomfree}]
The proof will use Theorem \ref{borthom}, Theorem \ref{trivthom}, and the following homotopy pushout square:
\begin{center}
	\begin{equation}\label{homsq1}
	\begin{tikzcd}
	(EC_2\times (X^{C_2})^E)_+ \arrow[r] \arrow[d]& (X^{C_2})^E\arrow[d]\\
	(EC_2\times X^E)_+ \arrow[r] & X^E
	\end{tikzcd}
	\end{equation}
\end{center}
One can check this is a homotopy pushout square by considering the maps on the underlying spaces and then on the fixed sets. Using the square above, we get the following long exact sequence from applying reduced Bredon cohomology:
	\[
	\to H^{p-1,q}(EC_2\times (X^{C_2})^E) \to \H^{p,q}(X^E) \to \H^{p,q}( (X^{C_2})^E) \oplus H^{p,q}(EC_2\times X^E) \to 
	\]
By definition, the Bredon cohomology of $EC_2\times X^E$ is the bigraded Borel cohomology of $X^E$. By Theorem \ref{borthom} there is a Thom class $u_E^{Bor}\in H^{n,0}_{Bor}(X^E)$. We can multiply this class by powers of $\tau$ to get elements $\tau^{q}u_E^{Bor}\in H^{n,q}_{Bor}(X^E)$. Note $\tau$ acts invertibly on Borel cohomology, so $\tau^qu_E^{Bor}$ satisfies the same properties as $u_E^{Bor}$.  

If $X$ is free, then $X^{C_2}$ is empty and $(X^{C_2})^E=pt$. The map $H^{p,q}(EC_2\times X^E) \to H^{p,q}(EC_2\times pt)$ is surjective and by restricting to fibers, one can check the Thom class $u_E^{Bor}$ is sent to zero. Thus $\H^{n,q}(X^E) \to H^{n,q}(EC_2\times X^E)$ is injective and there is a unique lift of $u_E^{Bor}$. Let $u_{E,q}$ be the pullback of $\tau^q u_{E}^{Bor}$ under this isomorphism. Note if $q<0$, we can still make sense of this because $\tau$ acts invertibly on Borel cohomology. The class $u_{E,q}$ then satisfies the properties of the theorem due to the properties of $\tau^qu_{E}^{Bor}$. This completes the proof of Theorem \ref{thomfree}.

If $X$ is nonfree, then we need to consider $(X^{C_2})^E$. Recall $X_1, \dots, X_m$ are the connected components of $X^{C_2}$. Let $E_i$ be the restriction of $E$ to $X_i$. Then $(X^{C_2})^E=X_1^{E_1}\vee \dots \vee X_k^{E_k}$ and 
	\[
	\H^{*,*}((X^{C_2})^E) \cong \H^{*,*}(X_1^{E_1})\oplus \dots \oplus \H^{*,*}(X_k^{E_k})
	\]
Each component $X_i$ is a connected and finite CW-complex, so we can apply Lemma \ref{trivthom} to get Thom classes $u_{E_i}\in \H^{n,q_i}(X_i^{E_i})$. Let $u_{E}^{fix}=\sum_{i=1}^k \tau^{q-q_i}u_{E_i}$ where recall $q$ is the largest of the $q_i$. 

Roughly speaking, we now paste together the classes $\tau^qu_{E}^{Bor}$ and $u_{E}^{fix}$ to get the class $u_{E}$ described in Theorem \ref{thomnonfree}. To be more precise, let's consider the following portion of the long exact sequence described at the beginning of the proof:
	\[
	 H^{n-1,q}(EC_2\times (X^{C_2})^E) \to \H^{n,q}(X^E) \to \H^{n,q}( (X^{C_2})^E) \oplus H^{n,q}(EC_2\times X^E) 
	\]
By the Borel Thom isomorphism theorem, 
	$$\H^{n-1,q}_{Bor}((X^{C_2})^E)=\H^{n-1,q}_{Bor}(X^E)=0.$$ 
These groups aren't zero when working unreduced because the cohomology of a point is $\Z/2[\rho,\tau^{\pm 1}]$, but the map $H^{n-1,q}_{Bor}( X^E) \to H^{n-1,q}_{Bor}((X^{C_2})^E)$ will be an isomorphism (this is just the cohomology of a point mapping to itself). Thus the left map in the above sequence must be zero and the right map is injective. Let's extend the long exact sequence to the right to get the following:
	\[
	\H^{n,q}(X^E) \hookrightarrow \H^{n,q}( (X^{C_2})^E) \oplus \H^{n,q}((EC_2\times X^E)_+) \to  \H^{n,q}((EC_2\times (X^{C_2})^E)_+).
	\]
We next show $u_{E}^{fix}+\tau^qu_{E}^{Bor}$ maps to zero under the right map by restricting to the fibers. Let $x\in E_{i}$ and let $\hat{E_x}$ denote the one-point compactification of the fiber $E_x$. We have a homotopy pushout square as shown below, and the inclusion $\hat{E_x}\hookrightarrow X^E$ gives a map from the square below to the square shown in \ref{homsq1}.
\begin{center}
\begin{equation}\label{homsq2}
	\begin{tikzcd}
	(EC_2\times \hat{E_{x}})_+ \arrow[r] \arrow[d]&  \hat{E_{x}}\arrow[d]\\
	(EC_2\times  \hat{E_{x}})_+ \arrow[r] & \hat{E_{x}}
	\end{tikzcd}
\end{equation}
\end{center}
Let's consider the long exact sequence for the above square as well as the induced map to square shown in \ref{homsq1}. We get the following commutative square:
\begin{center}
\begin{equation}\label{commfiber}
	\begin{tikzcd}
	\H^{n,q}( (X^{C_2})^E) \oplus \H^{n,q}((EC_2\times X^E)_+) \arrow[r]\arrow[d]&\H^{n,q}((EC_2\times (X^{C_2})^E)_+)\arrow[d,"\cong"]\\
	\H^{n,q}(\hat{E_x}) \oplus \H^{n,q}((EC_2\times  \hat{E_{x}})_+ ) \arrow[r]&\H^{n,q}((EC_2\times \hat{E_x})_+)\
	\end{tikzcd}
\end{equation}
\end{center}
The sum of the restrictions of $u_{E}^{fix}$ and $\tau^qu_{E}^{Bor}$ maps to zero in the down-across route (this can be seen by forgetting to singular cohomology, for example), and thus $u_{E}^{fix}+\tau^qu_{E}^{Bor}$ maps to zero in the top row. We can now define the class $u_{E}$ in $\H^{n,q}(X^E)$ to be the unique pullback of this sum. 

To finish the proof, we just need to check $u_E$ satisfies the various properties described in the theorem. Most of these properties follow immediately from properties of $u_E^{fix}$ and $u_E^{Bor}$. For property (i), consider the diagram labeled \ref{commfiber}. Note $u_E^{fix}\mapsto \tau^{q-q_i}\alpha_x$ where $\alpha_x$ is the generator of $\H^{n,q}(\hat{E_x})$. If we extend the rows of this diagram to the left, then we see $u_E\mapsto \tau^{q-q_i}\alpha_x$ under the restriction as well. We can similarly show property (ii) by considering the fibers $E_{x,\sigma x}$ and using $u_E^{Bor}$. To prove property (iii), we consider the following map of long exact sequences:
\begin{center}
\begin{tikzcd}[column sep=small]
	 H^{f-1+n,g+q}(EC_2\times (X^{C_2})^E) \arrow[r]& \H^{f+n,g+q}(X^E) \arrow[r]&  \begin{array}{c}\H^{f+n,g+q}( (X^{C_2})^E) \\ \oplus \\ H^{f+n,g+q}(EC_2\times X^E)\end{array}\\
	  H^{f-1,g}(EC_2\times X^{C_2}) \arrow[r]\arrow[u,"\smile u_{E|_{X^{C_2}}}^{Bor}"]& \H^{f,g}(X) \arrow[r]\arrow[u,"\smile u_E"]& \H^{f,g}( X^{C_2}) \oplus H^{f,g}(X)\arrow[u,"\smile u_E^{fix}\oplus \smile u_E^{Bor}"]
\end{tikzcd}
\end{center}

By the five-lemma, the center vertical map will be an isomorphism if the previous two and following two vertical maps are isomorphisms. Multiplication by the Borel Thom class always gives an isomorphism, but issues arise when considering $u_E^{fix}$. Recall the class $u_{E}^{fix}$ was defined to be the sum $\sum_{i=1}^k \tau^{q-q_i}u_{E_i}$ and $\tau$ does not act invertibly on $H^{*,*}((X^{C_2})^E)$. Though, the map $H^{f,g}((X^{C_2})^E)\to H^{f,g+1}((X^{C_2})^E)$ given by multiplication by $\tau$ is an isomorphism whenever $g\geq f$ by Lemma \ref{tauiso}. Note if $g\geq f$, then $g\geq f-1$, so we can apply the five lemma to see multiplication by $u_E$ is an isomorphism whenever $g\geq f$. This proves property (iii). If $q_i=q$ for all i, then cupping with $u_E^{fix}$ is an isomorphism in all ranges, and this gives property (iv). This completes the proof of Theorem \ref{thomnonfree}.
\end{proof}

Part I of the theorems proved the defining properties of Thom classes in Bredon cohomology. In the second part we prove additional properties of the Thom classes and of the cohomology of Thom spaces.

\begin{theorem}[Thom isomorphism for nonfree base spaces, Part II]\label{thomnonfreeII} Let $X$, $\pi:E\to X$, $n$, $X_i$, $q_i$, $q$ and $u_E\in H^{n,q}(E,E')$ be defined as in Theorem \ref{thomnonfree}. The following additional properties hold:

\begin{enumerate}[label=(\roman*)]
	\item[(v)] $\psi(u_E)$ is the singular Thom class, where $\psi:H^{n,q}(E,E')\to H^n_{sing}(E,E')$ is the forgetful map;
	\item[(vi)] $\M_2\cdot u_E \cong \Sigma^{n,q}\M_2$ where $\M_2\cdot u_E$ denotes the submodule generated by $u_E$;
	\item[(vii)] Suppose 
		\[
		H^{*,*}(X) \cong (\oplus_{i=1}^c \Sigma^{k_i,\ell_i}\M_2) \oplus (\oplus_{j=1}^d \Sigma^{s_j,0}A_{r_j}).
		\] 
		Then		
		\[
		H^{*,*}(E,E') \cong (\oplus_{i=1}^c \Sigma^{k_i+n,\ell'_i}\M_2) \oplus (\oplus_{j=1}^d \Sigma^{s_j+n,0}A_{r_j})
		\]  
	 where the weights $\ell_i'$ satisfy $\ell_i+q\geq \ell_i'\geq 0$.
\end{enumerate}
\end{theorem}

\begin{proof}
For properties (v) and (vi), let $x\in X^{C_2}$ such that $E_x\cong \R^{n,q}$. By restricting to this fiber and forgetting to singular cohomology, we get the following commutative diagram:
	\begin{center}
	\begin{tikzcd}
	H^{n,q}(E,E') \arrow[r] \arrow[d,"\psi"]& H^{n,q}(E_x, E_x')\arrow[d,"\psi", "\cong" left]\\
	H^{n}_{sing}(E,E') \arrow[r] & H^{n}_{sing}(E_x, E_x')
	\end{tikzcd}
	\end{center}
The top horizontal map takes $u_E$ to the generator, and thus $\psi(u_E)$ is nonzero. This holds for all fibers, and so it must be that $\psi(u_E)$ is the singular Thom class; this shows property (v). For (vi), note $\theta\cdot u_E$ must be nonzero because of the top map. In \cite[Lemma 4.1]{M1}, it is shown the class $\theta$ detects copies of $\M_2$, which proves property (vi).

All that remains is property (vii). This property relies on a somewhat technical proof about ``nice" $\M_2$-modules. This is given in Appendix \ref{ch:algproof}.
\end{proof}

\begin{theorem}[Thom isomorphism for free base spaces, Part II]\label{thomfreeII} Let $X$, $\pi:E\to X$, and $u_{E,q}\in H^{n,q}(E,E')$ be defined as in Theorem \ref{thomfree}. The following additional properties hold:
\begin{enumerate}[label=(\roman*)]
	\item[(iv)] $\psi(u_{E,q})$ is the singular Thom class, where $\psi:H^{n,q}(E,E')\to H^n_{sing}(E,E')$ is the forgetful map; and
	\item[(v)] For any $C_2$-CW complex $Y$ and map $Y\to X$, $f^*(u_{E,q})=u_{f^*E,q}$.
\end{enumerate}
\end{theorem}
\begin{proof}
Property (iv) can be proven just as property (v) in the previous theorem. For property (v), note $Y$ must also be free, and then the statement follows by restricting to fibers and using uniqueness of the Thom class for bundles over free $C_2$-CW complexes. 
\end{proof}

There is also a naturality statement for maps to nonfree $C_2$-CW complexes. We include this as a proposition below. 

\begin{prop}[Naturality of nonfree Thom class]\label{natnonfree} Let $X$ and $Y$ be finite $C_2$-CW complexes and suppose $X$ is nonfree. Let $E\to X$ be an $n$-dimensional $C_2$-vector with maximum weight $q$ over $X^{C_2}$. For any map $f:Y \to X$, there are two cases for $f^*u_{E}$:
\begin{enumerate}[label=(\roman*)]
\item If $Y$ is free, then $f^*u_E = u_{f^*E,q}$; and
\item If $Y$ is nonfree, then $f^*u_E = \tau^{q-q_Y} u_{f^*E}$ where $q_Y$ is the maximum weight of fibers over $Y^{C_2}$ in $f^*E$ (note $q_Y\leq q$). 
\end{enumerate}
\end{prop}
\begin{proof}
These statements follow by restricting to fibers and using the uniqueness of the Thom class. 
\end{proof}

We conclude this section by considering what happens to the nonfree Bredon Thom class under taking fixed points; this will be helpful in our discussion of fundamental classes that follows. Suppose $X$ is nonfree and as before, let $X_1,\dots, X_m$ be the connected components of $X^{C_2}$. Again suppose we have an $n$-dimensional $C_2$-vector bundle $E\to X$ such that the weight of the fibers over each $X_i$ is $q_i$. By taking fixed points of $E_i$, we obtain a vector bundle $E_i^{C_2} \to X_i$ of dimension $n-q_i$. The inclusion $E^{C_2} \hookrightarrow E$ induces a map 
	\[
	j:X_1^{E_1^{C_2}}\vee \dots \vee X_m^{E_m^{C_2}} \to X^E.
	\] 
\begin{prop}\label{fixpts}
Under the map $j^*:\H^{*,*}(X^E) \to  \bigoplus_{i=1}^m \H^{*,*}(X_m^{E_m^{C_2}})$ where $j$ is the map described above, we have $j^*(u_E) = \sum_{i=1}^m \tau^{q-q_i}\rho^{q_i}\cdot u_{E_i^{C_2}}$.
\end{prop}
\begin{proof} Let $x_i\in E_i$ and consider the restriction to this fiber and the restriction to the fixed sets. Note $H^{*,*}(E_{x_i},E_{x_i}')\cong \H^{*,*}(S^{n,q_i})$ and these restrictions fit together to give the following commutative diagram:
\begin{center}
	\begin{tikzcd}
	\H^{*,*}(X^E) \arrow[dr]\arrow[r]& \H^{*,*}((X^{C_2})^E) \arrow[r] \arrow[d]& \H^{*,*}((X^{C_2})^{E^{C_2}}) \arrow[d]\\
& \H^{*,*}(S^{n,q_i}) \arrow[r]&\H^{*,*}((S^{n,q_i})^{C_2}) = \H^{*,*}(S^{n-q_i,0})
	\end{tikzcd}
\end{center}
Let $\alpha \in \H^{n,q_i}(S^{n,q_i})$ and $\beta \in \H^{n-q_i,0}(S^{n-q_i,0})$ denote the respective generators of these rank one free modules. Let's analyze the bottom right map. Consider the cofiber sequence $(S^{n,q_i})^{C_2}\hookrightarrow S^{n,q_i} \to S^{n,q_i}/(S^{n,q_i})^{C_2}$. If the bottom horizontal map was zero, the associated long exact sequence to this cofiber sequence would imply $\H^{*,*}(S^{n,q_i}/(S^{n,q_i})^{C_2})$ is a rank two free module, but the space $S^{n,q_i}/(S^{n,q_i})^{C_2}$ has exactly one fixed point, so this contradicts Corollary \ref{freestructure}. Thus the bottom horizontal map is not zero and it must be that $\alpha \mapsto \rho^{q_i}\beta$.

To prove the proposition, we trace the Thom class $u_E\in \H^{*,*}(X^E)$ around the diagram. Under the left diagonal map $u_E\mapsto \tau^{q-q_i}\alpha$, which then under the bottom horizontal map goes to $\tau^{q-q_i}\rho^{q_i}\beta$. Under the upper left horizontal map $u_E\mapsto u_E^{fix}=\sum_{i=1}^m \tau^{q-q_i}u_{E_i}$ by construction. Thus $u_E^{fix}$ must map to a class in $ \H^{*,*}((X^{C_2})^{E^{C_2}})$ that restricts to $  \tau^{q-q_i}\rho^{q_i}\beta$. Since this must hold over every component, the only such class is $\sum_{i=1}^m \tau^{q-q_i}\rho^{q_i}u_{E_i^{C_2}}$. The composition of the upper horizontal maps is exactly $j^*$, so this completes the proof. 
\end{proof}


\section{Fundamental classes}\label{ch:funclasses}

We now employ Theorems \ref{thomnonfree} and \ref{thomfree} to define fundamental classes for $C_2$-submanifolds. We prove a handful of properties about these classes, such as how they forget to the usual singular fundamental classes, and how the product of these fundamental classes is given in terms of the intersection of the submanifolds in the transverse case. We end the section with a series of examples.

\subsection{Nonequivariant fundamental classes} 
Recall for singular cohomology, we can define fundamental classes using the Thom isomorphism theorem as follows. Let $X$ be a closed $n$-dimensional manifold and $Y\subset X$ be a closed $k$-dimensional submanifold. If we are working with $\Z/2$-coefficients, all vector bundles over $Y$ are orientable. In particular, if $\pi:N\to Y$ is the normal bundle of $Y$ in $X$, then the Thom isomorphism theorem guarantees a unique class $u\in H^{n-k}(N,N-0)$ known as the Thom class and
	\[
	\pi^*(-)\smile u: H^{j}(Y)\to H^{n-k+j}(N,N-0) 
	\]
is an isomorphism. There exists a tubular neighborhood $U$ of $Y$ in $X$, and by excision we have the following isomorphism
	\[
	H^{n-k}(N,N-0)\cong H^{n-k}(U,U-Y)\cong H^{n-k}(X,X-Y).
	\]
Thus there is a unique nonzero class in $H^{n-k}(X,X-Y)$ corresponding to the Thom class. We can now define $[Y]\in H^{n-k}(X)$ to be the image of this unique class under the induced map from the inclusion of pairs $(X,\emptyset)\hookrightarrow (X,X-Y)$. We will often denote these singular classes by $[Y]_{sing}$ to distinguish them from the Bredon cohomology fundamental classes defined below. 

\subsection{Definitions and first properties}
In Section \ref{ch:eqvb}, we proved properties of the cohomology of Thom spaces of $C_2$-vector bundles. These properties are enough to transfer the above definition from the singular world to the equivariant world. 

Let $X$ be a closed $C_2$-manifold and $Y$ be a closed $C_2$-submanifold. First suppose both $Y$ and $X$ are nonfree, and suppose topologically $Y$ is $k$-dimensional and $X$ is $n$-dimensional. Let $N\to Y$ be the normal bundle of $Y$ in $X$, and let $q$ be the maximum weight of $N$ over $Y^{C_2}$ as in Theorem \ref{thomnonfree}. By this theorem, we are guaranteed a unique class $u_N\in H^{n-k,q}(N,N-0)$. Let $U$ be an equivariant tubular neighborhood of $Y$. Using this neighborhood and excision
	\[
	H^{n-k,q}(N,N-0)\cong H^{n-k,q}(U,U-Y) \cong H^{n-k,q}(X,X-Y).
	\]
We are now ready for the following definition.

\begin{definition}\label{funclasses} 
Let $Y$, $X$, $n$, $k$, and $q$ be defined as above. The unique nonzero class in $H^{n-k,q}(X,X-Y)$ corresponding to the Thom class $u_N$ in the above isomorphism is denoted by $[Y]_{rel}$ and referred to as {\textbf{the relative fundamental class of $\mathbf{Y}$ in $\mathbf{X}$}}. The image of this class in $H^{n-k,q}(X)$ under the induced map by the inclusion of the pair $(X,\emptyset) \hookrightarrow (X,X-Y)$ will be denoted $[Y]$ and referred to as {\textbf{the fundamental class of $\mathbf{Y}$ in $\mathbf{X}$}}.
\end{definition}

The following corollary of property (v) of Theorem \ref{thomnonfreeII} relates these equivariant fundamental classes to the nonequivariant fundamental classes.

\begin{corollary}\label{forget1} 
Suppose $Y\subset X$ as above. Then $\psi([Y])=[Y]_{sing}$ where $\psi$ is the forgetful map to singular cohomology. 
\end{corollary}

For $Y\subset X$, let $Y^{C_2}=Y_1\sqcup \dots \sqcup Y_j$ where each $Y_i$ is a connected component. We have the following corollary to Proposition \ref{fixpts}.

\begin{corollary}\label{funfix} Under the map $i:X^{C_2}\to X$, $i^*([Y])=\sum_{i=1}^j \tau^{q-q_i}\rho^{q_i}[Y_i]$.
\end{corollary}

We can similarly define fundamental classes for free submanifolds. Let $X$ be an $n$-dimensional $C_2$-manifold and $Y\subset X$ be a $k$-dimensional $C_2$-submanifold with a free $C_2$-action. Again we can consider the normal bundle $N\to Y$ and from Theorem \ref{thomfree}, we have the Thom class $u_{N,q}\in H^{n-k,q}(N, N-0)$ for each weight $q$. We can now define the following.

\begin{definition} Let $X$, $Y$, $n$, $k$, and $N$ be as above. For every integer $q$, the unique element in $H^{n-k,q}(X,X-Y)$ corresponding to the Thom class $u_{N,q}$ is denoted $[Y]_{q,rel}$ and referred to as the {\textbf{relative fundamental class of $\mathbf{Y}$ in $\mathbf{X}$ of weight $\mathbf{q}$}}. The image of this relative class under the map induced by the inclusion of pairs $(X,\emptyset)\hookrightarrow (X,X-Y)$ is denoted $[Y]_q$ and referred to as the {\textbf{fundamental class of $\mathbf{Y}$ in $\mathbf{X}$ of weight $\mathbf{q}$}}.
\end{definition}

We have the following corollary of properties (ii) and (iv) of Theorems \ref{thomfree} and \ref{thomfreeII}.

\begin{corollary}\label{forget1} 
Suppose $Y\subset X$ as above. Then $\tau[Y]_q=[Y]_{q+1}$ and $\psi([Y]_q)=[Y]_{sing}$ where $\psi$ is the forgetful map to singular cohomology. 
\end{corollary}

\subsection{Naturality and product formulas} 
We now prove there are naturality and intersection product formulas for equivariant fundamental classes. The first theorem is a consequence of property (ii) of Theorem \ref{thomfree} and Lemma \ref{natnonfree}.

\begin{theorem}[Naturality of fundamental classes]\label{funnat} Let $X$ and $Y$ be $C_2$-manifolds and $Z\subset X$ be a $C_2$-submanifold of topological codimension $n-k$. Suppose $f:Y\to X$ is a smooth map that is transverse to $Z$. Then we have the following pullback formulas:
\begin{enumerate}[label=(\roman*)]
\item If $f^{-1}(Z)$ is free and $Z$ is nonfree with $[Z]\in H^{n-k,q}(X)$, then $f^*([Z])=[f^{-1}(Z)]_q$;
\item If $f^{-1}(Z)$ is nonfree with $[f^{-1}(Z)]\in H^{n-k,q'}(Y)$ and $Z$ is nonfree with $[Z]\in H^{n-k,q}(X)$, then $f^*([Z])=\tau^{q-q'}[f^{-1}(Z)]$; and
\item If $f^{-1}(Z)$ and $Z$ are free, then $f^*([Z]_q)=[f^{-1}(Z)]_q$ for all $q$.
\end{enumerate}
\end{theorem}

\begin{proof} Let $W=f^{-1}(Z)$. Denote the normal bundle for $Z$ in $X$ by $N_{Z}$ and  the normal bundle for $W$ in $Y$ by $N_{W}$. For each $w\in W$ the derivative of $f$ gives a map of tangent spaces $df_{w}:T_{w}Y\to T_{f(w)}X$. By definition of transversality, $df_w(T_{w}(W))=T_{f(w)}Z$. Thus we get a map on the quotient $df_w:(N_W)_w\to (N_Z)_{f(w)}$ which gives a map of bundles $df:N_W\to N_Z$. We have the following commutative square
	\begin{center}
	\begin{tikzcd}
	N_{W}\arrow[d] \arrow[r,"df"] & N_Z\arrow[d]\\
	W \arrow[r,"f"]&Z
	\end{tikzcd}
	\end{center}
which implies there is a map to the pullback $N_W\to f^*N_Z$. This is an isomorphism on all fibers and thus an isomorphism of vector bundles. 

Let $p_Z: X_{+}\to Z^{N_Z}$ denote the Pontryagin-Thom collapse map. An equivalent way to define the fundamental class of $Z$ is to pull back the Thom class for $N_Z$ under $p_Z$. Similarly we can define the fundamental class of $W$ using $p_{W}:Y_+\to W^{N_W}$. We have the following commutative diagram in the homotopy category:
	\begin{center}
	\begin{tikzcd}
	 Y_+\arrow[d,"f"] \arrow[r,"p_W"] &W^{N_W}\arrow[d,"\tilde{f}"]\arrow[r,"\cong"]& W^{f^{*}N_Z}\arrow[dl]\\
	 X_+ \arrow[r,"p_Z"]& Z^{N_Z} &
	\end{tikzcd}
	\end{center}
Consider cases (i), (ii), and (iii) in the theorem. The commutativity of the right triangle along with Lemma \ref{natnonfree} imply in case (i) $\tilde{f}^*(u_{N_Z}) = u_{N_W,q}$ and in case (ii) $\tilde{f}^{*}(u_{N_Z}) = \tau^{q-q'}u_{N_W}$. Part (iv) of Theorem \ref{thomfreeII} implies in case (iii) $\tilde{f}^*(u_{N_Z,q}) = u_{N_W,q}$. Pulling back to $X_+$ and $Y_+$ then completes the proof. 
\end{proof}

As in the nonequivariant case, we have a nice intersection product for equivariant fundamental classes. The various product formulas are given in the theorem below. 

\begin{theorem}[Intersection product]\label{prod}Let $X$ be a $C_2$-manifold and suppose $Y$, $Z$ are two $C_2$-submanifolds of $X$ that intersect transversally. Let $W=Y\cap Z$. The following formulas hold for the product of the fundamental classes for $Y$ and $Z$:
\begin{itemize}
\item[(i)] If $Y$ and $Z$ are free, then $[Y]_s\smile [Z]_r=[W]_{r+s}$.
\item[(ii)] If $Y$ and $Z$ are nonfree and $W$ is free, then $[Y]\smile [Z] = [W]_{q_Y+q_Z}$ where $q_Y$ and $q_Z$ are the weights of the classes $[Y]$ and $[Z]$, respectively.
\item[(iii)] If $Y$ is nonfree and $Z$ is free, then $[Y] \smile [Z]_r=[W]_{q_Y+r}$.
\item[(iv)] If $Y$, $Z$, and $W$ are nonfree, then $[Y]\smile [Z]=\tau^{q_Y+q_Z-q_W}[W]$.
\end{itemize}
\end{theorem}
\begin{proof}
As in the previous proof, let $A^{N_A}$ denote the Thom space for the normal bundle of a submanifold $A$ in $X$ and $p_A: X_+ \to A^{N_A}$ denote the Pontryagin-Thom collapse map. Observe $N_Y|_W\oplus N_Z|_W \cong N_W$ by the transversality assumption. Let $N_Y|_W\oplus^eN_Z|_W$ denote the external direct sum over $W\times W$ and consider the following commutative diagram in the homotopy category:
\begin{center}
\begin{tikzcd}
X_+\arrow[r,"\Delta"]\arrow[d, "p_W"left] & X_+\wedge X_+ \arrow[r,"p_Y\wedge p_Z"]&Y^{N_Y}\wedge Z^{N_Z}\\
W^{N_W}\cong W^{N_Y|_W\oplus N_Z|_W} \arrow[r,"j_*"]& (W\times W)^{N_Y|_W\overset{e}{\oplus} N_Z|_W}\arrow[r,"\cong"]& W^{N_Y|_W}\wedge W^{N_Z|_W}\arrow[u,"(i_Y)_*\wedge (i_Z)_*" right]
\end{tikzcd}
\end{center}
In the above, the maps $(i_Y)_*$ and $(i_Z)_*$ are induced by the inclusions, and $j_*$ is the map induced by the diagonal map $W \to W \times W$. The homeomorphism shown on the bottom right can be described as follows. On the left, a point in the domain can be written as a pair $(w_1,w_2, \alpha_1+\alpha_2)$ where $w_i\in W$ and $\alpha_i\in (N_Y|_W)_{w_i}$. This pair is then sent to $(w_1, \alpha_1) \wedge (w_2,\alpha_2)$.

Now assume $Y$ has codimension $j$ and $Z$ has codimension $k$. Applying Bredon cohomology to the above and incorporating the external product gives the following commutative diagram:
\begin{center}
\begin{tikzcd}
&\H^{j,*}(X_+) \otimes \H^{k,*}(X_+) \arrow[ld] \arrow[d,"\mu"]& \arrow[l]\H^{j,*}(Y^{N_Y})\otimes \H^{k,*}(Z^{N_Z})\arrow[d,"\mu"]\\
\H^{j+k,*}(X_+) &\arrow[l,"\Delta_*"] \H^{j+k,*}(X_+\wedge X_+)&\arrow[l,"(p_Y\wedge p_Z)_*"]\H^{j+k,*}(Y^{N_Y}\wedge Z^{N_Z})\arrow[d,"(i_Y)_*\wedge (i_Z)_*" right]\\
\H^{j+k,*}(W^{N_W}) \arrow[u, "(p_W)_*"left]&\arrow[l,"j_*"] \H^{j+k,*}(W^{N_Y|_W\overset{e}{\oplus} N_Z|_W})& \arrow[l,"\cong"]\H^{j+k,*}(W^{N_Y|_W}\wedge W^{N_Z|_W})
\end{tikzcd}
\end{center}

To prove the various product formulas, we begin in the upper right corner and consider the image of the Thom classes as we go different ways around the diagram. We explain how to do this in case (iv) noting the other cases follow similarly. Consider $u_{N_Y}\otimes u_{N_Z}$. Going to the left and diagonally down yields the product $[Y] \smile [Z]$. Consider instead going down and over to the bottom left corner. We claim $u_{N_Y}\otimes u_{N_Z}\mapsto \tau^{q_Z+q_Y-q_W}u_{N_W}$. To see this, first observe $q_W\leq q_Z+q_Y$ because $N_W\cong N_Y|_W\oplus N_Z|_W$ and recall for a nonfree submanifold $A$, the integer $q_A$ is the largest weight representation appearing over the fixed set of $A$ in the normal bundle $N_A$. If we restrict to any fiber over a point in $W$, we see the class $u_{N_Y}\otimes u_{N_Z}$ must map to something that restricts to the nonzero class over the fiber. By the uniqueness of Thom classes, it must be that $u_{N_Y}\otimes u_{N_Z}$ maps to $\tau^{q_Y+q_Z-q_W}u_{N_W}$ in the bottom left corner. Mapping to the cohomology of $X$ then yields the desired formula. 
\end{proof}

\subsection{Examples} We finish this section with a series of examples.

\begin{example}[Fundamental class of a fixed point]
Suppose $X$ is a nonfree, closed, connected $C_2$-manifold and $x\in X^{C_2}$ is a fixed point. Let $D$ be a tubular neighborhood of this fixed point, so $D\cong \R^{n,k}$ for some $k$. Then there is a class $[x]\in H^{n,k}(X)$, and since this class forgets to the singular class $[x]\in H^{n}_{sing}(X)$, it is necessarily nonzero.

The classes for fixed points provide some insight into Theorem \ref{topm2}. Choose a point $x\in X^{C_2}$ that is in a component of $X^{C_2}$ of the smallest topological codimension $k$. Let $D\cong \R^{n,k}$ be a neighborhood of this point. It was shown in \cite[Corollary A.2]{Ha} that the map $q:X\to X/(X\setminus D)\cong S^{n,k}$ induces a split injection on Bredon cohomology. By Theorem \ref{funnat}, $q^*([q(x)])=[x]$. Now the class $[q(x)]\in H^{*,*}(S^{n,k})\cong \M_2 \oplus \Sigma^{n,k}\M_2$ forgets to something nonzero, so it must generate the free summand in bidegree $(n,k)$. Thus the class $[x]$ generates a free summand in bidegree $(n,k)$ in $H^{*,*}(X)$.
\end{example}

\begin{example}\label{s11ex} Consider the one-dimensional $C_2$-manifold $S^{1,1}$ whose cohomology is shown in Figure \ref{fig:s11ex}.
\begin{figure}[ht]
\begin{tabular}{m{0.5\textwidth} m{0.5\textwidth}}
\begin{center}
	\begin{tikzpicture}
		\draw (0,0) circle (1cm);
		\draw[blue] (0,1) node{\dot};
		\draw[blue] (0,1.3) node{\small{$a$}};
		\draw[blue] (0,-1) node{\dot};
		\draw[blue] (0,-1.3) node{\small{$b$}};
		\draw[<->] (-0.2,1.6)--(.2,1.6);
		\draw[white] (0,-2.5)--(.1,-2);
		\draw (0,-2) node{$S^{1,1}$};
	\end{tikzpicture}
\end{center}
&
	\begin{tikzpicture}[scale=.4]
		\draw[help lines,gray] (-5.125,-5.125) grid (5.125, 5.125);
		\draw[<->] (-5,0)--(5,0)node[right]{$p$};
		\draw[<->] (0,-5)--(0,5)node[above]{$q$};
		\cone{-0.1}{0}{black};
		\cone{1.1}{1}{black};
	\draw (0,-6.5)node{$H^{*,*}(S^{1,1})$};
	\end{tikzpicture}
\end{tabular}
\caption{The space $S^{1,1}$ and its cohomology.}\label{fig:s11ex}
\end{figure}
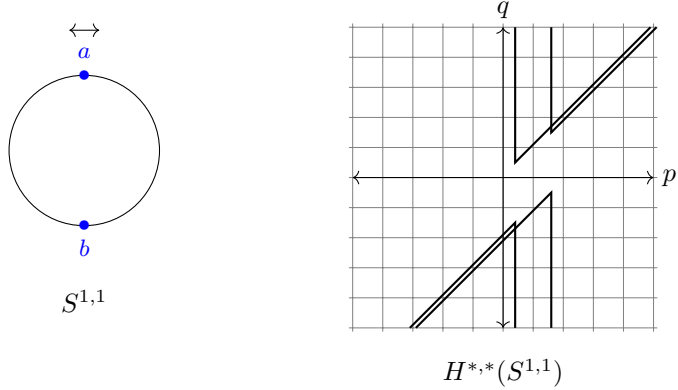

There are two fixed points, so there are two fundamental classes $[a]$, $[b]\in H^{1,1}(S^{1,1})$. By the previous example, $[a]$ and $[b]$ are both nonzero. We thus have three nonzero elements in $H^{1,1}(S^{1,1})$, namely $[a]$, $[b]$, and $\rho\cdot 1$. We would like to determine the dependence relation between these classes. 

Since $\psi([a])\neq 0$ and $\psi([b])\neq 0$, $[a]\neq \rho$ and $[b]\neq \rho$. From Corollary \ref{funfix}, under the map $i^*:\{a,b\} \hookrightarrow S^{1,1}$, $i^*([a])=\rho[a]$ and $i^*([b])=\rho[b]$. These classes are not equal in the cohomology of $\{a,b\}$, and thus not equal in the cohomology of $S^{1,1}$. Since we are working over $\Z/2$, the only possibility for a dependance relation is
	\[
	[a]+[b]+\rho = 0.
	\]
By Theorem \ref{prod}, $[a]\cdot[b] = [\{a\}\cap \{b\}]=0$. Using the dependance relation above we obtain $[a]^2=[a]([b]+\rho)=\rho[a]$, and we have recovered the following isomorphism of $\M_2$-algebras:
	\[
	H^{*,*}(S^{1,1})\cong \M_2[x]/(x^2=\rho x),~~|x|=(1,1).
	\]
\end{example}

\begin{example}\label{trefl} 
Consider the $C_2$-surface $X\cong S^{2,1}+[S^{1,0}-AT]$ which can be depicted as a torus with a reflection action as shown in Figure \ref{fig:trefl}. The fixed set is shown in {\textbf{\color{blue}{blue}}}. By Theorem \ref{nonfreeanswer}, the cohomology of $X$ is
	\[
	H^{*,*}(X)\cong \M_2 \oplus \Sigma^{1,0}\M_2 \oplus \Sigma^{1,1}\M_2 \oplus \Sigma^{2,1}\M_2
	\]
as shown on the grid.
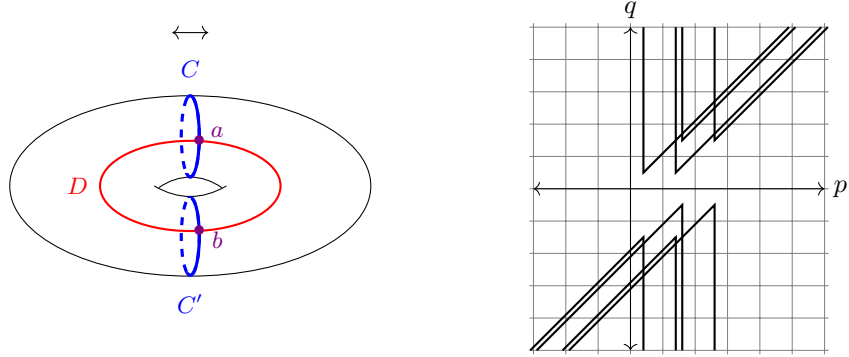
\begin{figure}[ht]
\begin{tabular}{m{0.5\textwidth} m{0.5\textwidth}}
\begin{center}
\begin{tikzpicture}[scale=1.2]
	\draw[<->] (-.2,1.7)--(.2,1.7);
	\draw (0,0) ellipse (2cm and 1cm);
	\draw (-0.4,0) to[out=330,in=210] (.4,0) ;
	\draw (-.35,-.02) to[out=35,in=145] (.35,-.02);
	\draw[very thick,blue] (0,.55) ellipse (.1cm and 0.45cm);
	\draw[white,fill=white] (-.3,0.15) rectangle (0,.96);
	\draw[very thick,blue,dashed] (0,.55) ellipse (.1cm and 0.45cm);
	\draw[very thick,blue] (0,-.56) ellipse (.1cm and 0.43cm);
	\draw[white,fill=white] (-.3,-0.95) rectangle (0,-.14);
	\draw[very thick,blue,dashed] (0,-.56) ellipse (.1cm and 0.43cm);
	\draw[thick,red] (0,0) ellipse (1cm and 0.5cm);
	\draw[violet] (0.1,.5)node{\dot};
	\draw[violet] (0.3,.6)node{\small{$a$}};
	\draw[violet] (0.1,-.5)node{\dot};
	\draw[violet] (0.3,-.6)node{\small{$b$}};
	\draw[blue] (0,1.3) node{\small{$C$}};
	\draw[blue] (0,-1.3) node{\small{$C'$}};
	\draw[red] (-1.25,0) node{\small{$D$}};
\end{tikzpicture}\end{center}&
\begin{center}
\begin{tikzpicture}[scale=.43]
\draw[help lines,gray] (-3.125,-5.125) grid (6.125, 5.125);
\draw[<->] (-3,0)--(6,0)node[right]{$p$};
\draw[<->] (0,-5)--(0,5)node[above]{$q$};
	\cone{-0.1}{0}{black};
	\cone{1.1}{1}{black};
	\cone{0.9}{0}{black};
	\cone{2.1}{1}{black};
\end{tikzpicture}
\end{center}
\end{tabular}
\caption{The space $X$ and its cohomology.}\label{fig:trefl}
\end{figure}

Notice $[C]$, $[C']\in H^{1,1}(X)$, $[D]\in H^{1,0}(X)$, and $[a]$, $[b] \in H^{2,1}(X)$. All of these fundamental classes forget to nonzero singular classes, so they must be nonzero. Now in $H^{1,1}(X)$, we have four nonzero elements $[C]$, $[C']$, $\tau[D]$, and $\rho$. Let's determine the dependence relations between these classes.

The classes $[C]$, $[C']$, and $\tau[D]$ forget to nonzero classes and so cannot equal $\rho$. Since $C$ and $C'$ are in different components of the fixed set, we can conclude from Corollary \ref{funfix} that they restrict to different classes under the inclusion of the fixed set. Thus $[C]\neq [C']$. If we forget to singular cohomology then $[C]_{sing}=[C']_{sing}$, so by the forgetful long exact sequence, it must be that $[C]+[C']=\rho$. We conclude $\{[C], \rho, \tau[D]\}$ forms a basis for $H^{1,1}(X)$. 

Using Theorem \ref{prod}, we have the following multiplicative relations
	\[
	[C][C']=0,~~[C][D]=[a],~~[C'][D]=[b],~~~[C]^2=[C]([C']+\rho)=\rho\cdot[C].
	\]
Let $x=[C]$ and $y=[D]$. We have shown as an $\M_2$-algebra
	\[
	H^{*,*}(X)\cong \M_2[x,y]/(x^2=\rho x, y^2=0),~~~|x|=(1,1),|y|=(1,0).
	\]
\end{example}

\begin{example}\label{rp2twex}
Consider the $C_2$-surface $Y$ whose underlying space is the projective plane. To describe the action, recall we can form $\R P^2$ by identifying antipodal points on the boundary of the disk. The space $Y$ will inherit the action from the rotation action on the disk, as depicted in Figure \ref{fig:rp2twex}. The fixed set is again shown in blue. The cohomology of this space is
	\[
	H^{*,*}(Y) \cong \M_2 \oplus \Sigma^{1,1}\M_2 \oplus \Sigma^{2,1}\M_2.
	\]
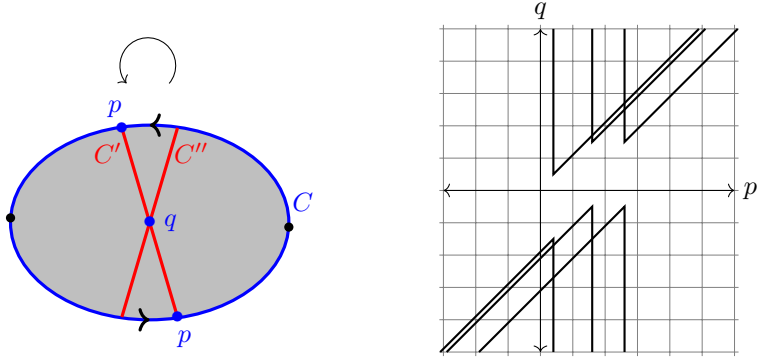
\begin{figure}[ht]
	\begin{subfigure}[b]{0.45\textwidth}
	\centering
	\begin{tikzpicture}[scale=1.85]
		\draw[blue,very thick, fill=lightgray, decoration={markings, mark=at position 0 with {\arrow[black,scale=.5]{*}}, mark=at position 0.25 with {\arrow[black]{>}}, mark=at position 0.5 with {\arrow[black,scale=.5]{*}}, mark=at position 0.75 with {\arrow[black]{>}}}, postaction={decorate}] (0,0) ellipse (1 cm and .7 cm);
		\draw[red, very thick] (-.2,.68)--(.2,-.68);
		\draw[blue] (-.2,.68) node{$\bullet$};
		\draw[blue] (-.25,.82) node{$p$};
		\draw[red] (-.3,.5) node{$C'$};
		\draw[red, very thick] (.2,.68)--(-.2,-.68);
		\draw[blue] (.2,-.68) node{$\bullet$};
		\draw[blue] (.25,-.84) node{$p$};
		\draw[red] (.3,.5) node{$C''$};
		\draw[blue] (0,0)node {$\bullet$};
		\draw[->] (.14,1) arc (-40:220:.20cm);
		\draw[blue] (.15,0) node{$q$};
		\draw[blue] (1.1,.15) node{$C$};
	\end{tikzpicture}
	\end{subfigure}
	\begin{subfigure}[b]{0.45\textwidth}
	\centering
	\begin{tikzpicture}[scale=.43]
		\draw[help lines,gray] (-3.125,-5.125) grid (6.125, 5.125);
		\draw[<->] (-3,0)--(6,0)node[right]{$p$};
		\draw[<->] (0,-5)--(0,5)node[above]{$q$};
		\cone{-0.1}{0}{black};
		\cone{1.1}{1}{black};
		\cone{2.1}{1}{black};
	\end{tikzpicture}
	\end{subfigure}
\caption{The space $Y$ and its cohomology.}\label{fig:rp2twex}
\end{figure}

Let's consider the submanifold $C$ and its normal bundle $N_C$. The circle $C$ is fixed, and for every $x\in N_C$, the fiber is given by $(N_C)_x\cong \R^{1,1}$. Thus we have a class $[C]\in H^{1,1}(Y)$. The submanifold $C'$ has two fixed points $p$, $q$, and $(N_{C'})_p\cong \R^{1,0}$ while $(N_{C'})_q\cong \R^{1,1}$. Thus we also have a class $[C']\in H^{1,1}(Y)$ and similarly a class $[C'']\in H^{1,1}(Y)$. It is clear $[C']=[C'']$. By considering neighborhoods, we also see $[p]\in H^{2,1}(Y)$ while $[q]\in H^{2,2}(Y)$. 

We have the following multiplicative relations from Theorem \ref{prod}:
	\[
	[C][C']=\tau\cdot[C\cap C']=\tau[p],\quad [C']^2=[C'][C'']=[C'\cap C'']=[q], \quad  \text{ and}
	\]
	\[
	[p][C']=[p][C'']=[\{p\}\cap C'']=0.
	\]
Under the inclusion of fixed points $i:Y^{C_2}\hookrightarrow Y$, $i^*(\tau[p])\neq i^*([q])$ by Corollary \ref{funfix}. We show $[q]=\tau[p]+\rho[C']$.

Note $[C]$ and $[C']$ both forget to the same nonequivariant nonzero class, so in particular, $[C]\neq\rho$ and $[C']\neq \rho$. Also since $[C][C']=\tau[p]\neq[q]$ while $[C']^2=[q]$, we see that $[C]\neq[C']$. Thus it must be that $[C]=[C']+\rho$, and this also shows $[q]=[C']^2=[C']([C]+\rho)=\tau[p]+\rho[C']$. 

Taking $x=[C']$, $y=[p]$, we can now state the cohomology of $Y$ as an $\M_2$-algebra:
	\[
	H^{*,*}(Y)\cong \M_2[x,y]/(x^2=\tau y+\rho x, y^2=0, xy=0),~~|x|=(1,1), |y|=(2,1).
	\]
\end{example}

\begin{example}[Fundamental class of conjugate points]\label{freepoint}
Let $X$ be a closed, connected, $n$-dimensional $C_2$-manifold with a non-fixed point $x\in X$. Consider the set of conjugate points $\{x,\sigma x\}$ and note this is isomorphic to the free orbit $C_2$. We show if $X$ is free, then $[x,\sigma x]_q\neq 0$ for all $q$. If $X$ is nonfree, let $p\in X^{C_2}$ be a fixed point whose fundamental class generates a free summand in bidegree $(n,k)$. We show $[x,\sigma x]_q\neq 0$ only for $q\leq k-2$, and explicitly, $[x,\sigma x]_q=\frac{\theta}{\tau^{q-k+2}}[p]$.

If $X$ is free, then the quotient map $\pi:X\to X/C_2$ is a map of smooth manifolds. By Theorem \ref{funnat}, $\pi^*([\pi(x)])=[x,\sigma x]_0$ and this class is nonzero as described in the proof of Lemma \ref{freefun}. 

For the nonfree case, we use the long exact sequence for the pair $(X, X-\{x , \sigma x\})$ to determine when the class is nonzero. Observe the space $X-\{x,\sigma x\}$ is a punctured $n$-manifold, so $H^{j}_{sing}(X-\{x,\sigma x\})=0$ for $j\geq n$. Using the forgetful long exact sequence, we see that 
	\[
	\rho:H^{j,q}(X-\{x,\sigma x\}) \to H^{j+1,q+1}(X-\{x,\sigma x\})
	\]
must be an isomorphism whenever $j\geq n$ and surjective when $j=n-1$. If $X$ is nonfree, then $X-\{x,\sigma x\}$ is also nonfree. Considering the structure theorem and the properties of the $\rho$ action, it must be that no summands in $H^{*,*}(X-\{x,\sigma x\})$ are generated in topological dimension $j$ for $j\geq n$ and all antipodal summands $\Sigma^{s,0}A_r$ must be concentrated in topological dimension less than $n$. Thus there is a sufficiently small $\ell$ such that $H^{n,q}(X-\{x,\sigma x\})=0$ whenever $q\leq \ell$. 

Fix $q$ such that $q\leq \ell$ and $q\leq k-2$. Consider the long exact sequence below:
\begin{center}
	\begin{tikzcd}
	\arrow[r]&H^{n,q}(X,X-\{x,\sigma x\}) \arrow[r]& H^{n,q}(X) \arrow[r] & H^{n,q}(X-\{x,\sigma x\})\arrow[r]&~
	\end{tikzcd}
\end{center}
The left group is $\Z/2$, and since the second map is surjective, there must be at most one nonzero element in $H^{n,q}(X)$. One such element is $\frac{\theta}{\tau^{q-k+2}}[p]$, and so the only option is $[x,\sigma x]_q = \frac{\theta}{\tau^{q-k+2}}[p]$. Now if this holds for some $q$, then it must hold for all $q\leq k-2$ by the action of $\tau$. Lastly for $q>k-2$, 
	\[
	[x,\sigma x]_{q}=\tau^{q-(k-2)}[x,\sigma x]_{k-2}=\tau^{q-(k-2)}\cdot \theta[p]=0.
	\]
\end{example}

\begin{example}\label{freetorex} 
Consider the free torus whose action is orientation reversing; this space was denoted $T^{anti}_1$ earlier in the paper. For notational simplicity, let $Z=T^{anti}_1$. An illustration of the space and its cohomology are shown in Figure \ref{fig:freetorex}. 
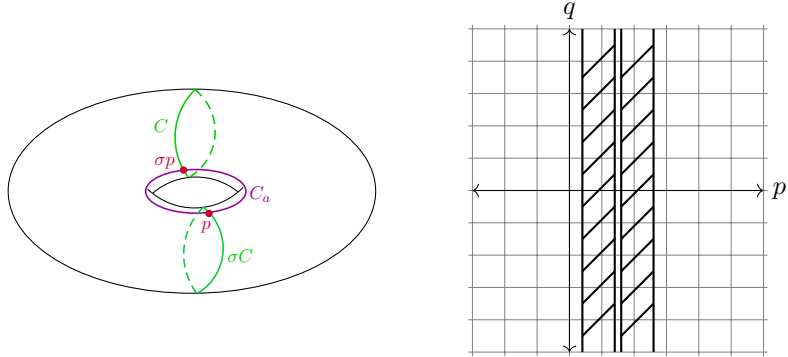
\begin{figure}[ht]
	\begin{subfigure}[b]{0.45\textwidth}
	\centering
	\includegraphics[scale=0.7]{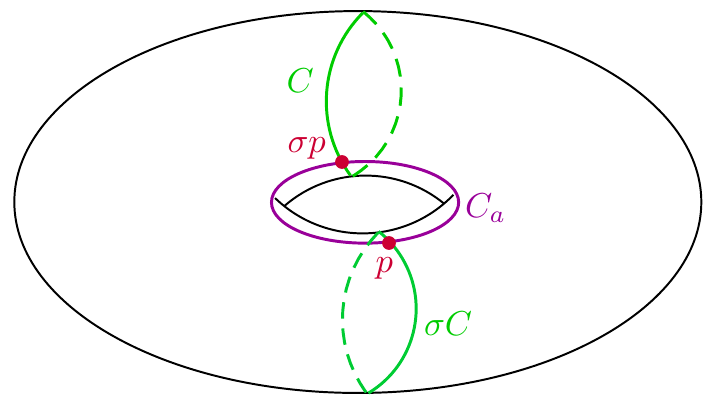}
	\vspace{0.3in}
	\end{subfigure}
	\begin{subfigure}[b]{0.45\textwidth}
	\centering
	\begin{tikzpicture}[scale=.43]
		\draw[help lines,gray] (-3.125,-5.125) grid (6.125, 5.125);
		\draw[<->] (-3,0)--(6,0)node[right]{$p$};
		\draw[<->] (0,-5)--(0,5)node[above]{$q$};
		\anti{-0.1}{1}{black};
		\anti{1.1}{1}{black};
	\end{tikzpicture}
	\end{subfigure}
	\caption{The space $T_1^{anti}$ and its cohomology.}
	\label{fig:freetorex}
\end{figure}

There are four families of cohomology classes of interest: $[C_a]_q$, $[C\sqcup\sigma C]_q$, \linebreak$[p\sqcup \sigma p]_q$, and $[Z]_q$. Observe the classes $[C_a]_q$ and $[Z]_q$ forget to nonzero classes in singular cohomology, so $[C_a]_q$ and $[Z]_q$ are nonzero for all $q$. Using the long exact sequence for the pair $(Z,Z-(C\sqcup \sigma C))$, one can check $[C\sqcup\sigma C]_q\neq 0$ for all $q$. Now $\psi([C\sqcup\sigma C]_q)=0$ so it must be that $[C\sqcup \sigma C]_q$ is in the image of $\rho$, and the only possibility is that $[C\sqcup \sigma C]_q = \rho \cdot [Z]_{q-1} = \rho\tau^{q-1}\cdot 1$. By perturbing $C$ we can find a submanifold $C'\sqcup \sigma C'$ such that $[C'\sqcup \sigma C']_q=[C\sqcup \sigma C]_q$ and the transverse intersection $(C'\sqcup \sigma C')\cap (C\sqcup \sigma C)$ is empty. 

From Theorem \ref{prod} we have that $[C_a]_r\cdot[C\sqcup \sigma C]_s=[p\sqcup \sigma p]_{r+s}$ for all $r$, $s$. This also follows from the module structure and the above fact that $[C\sqcup\sigma C]_1=\rho$, but it is nice to be able to recover the relation using fundamental classes.

The action of $\tau$ on the cohomology of $X$ is invertible, and it is easier to describe the cohomology as a $\tau^{-1}\M_2$-algebra. Note this also encodes the $\M_2$-algebra structure. As a $\tau^{-1}\M_2$-algebra we have recovered the following isomorphism where $x=[C_a]$:
	\[
	H^{*,*}(T^{anti}_1)\cong \tau^{-1}\M_2[x]/(\rho^2\cdot 1=0, x^2=0),~|x|=(1,1).
	\]
\end{example}

\begin{example}\label{tspit}
We do one more example that has both nonfree and free fundamental classes. Consider $X=S^{2,2}\#_2 T_1$ which can be depicted as a genus two torus with a rotation action, as shown below. The cohomology of this space is given by
	\[
	H^{*,*}(X) \cong \M_2 \oplus \left(\Sigma^{1,0}A_0\right)^{\oplus 2} \oplus \Sigma^{2,2}\M_2
	\]
which is shown in the grid below. 
\begin{figure}[ht]
	\begin{subfigure}[b]{0.45\textwidth}
	\centering
		\includegraphics[scale=0.77]{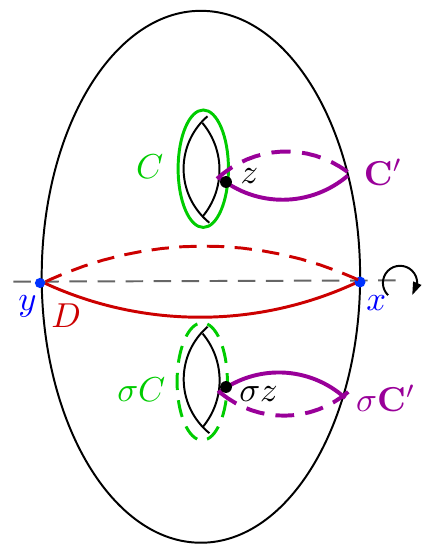}
	\end{subfigure}
	\begin{subfigure}[b]{0.45\textwidth}
	\centering
	\begin{tikzpicture}[scale=.43]
		\draw[help lines,gray] (-3.125,-5.125) grid (6.125, 5.125);
		\draw[<->] (-3,0)--(6,0)node[right]{$p$};
		\draw[<->] (0,-5)--(0,5)node[above]{$q$};
		\cone{0}{0}{black};
		\cone{2.1}{1.9}{black};
		\lab{1.1}{2}{black};
		\anti{1.1}{0}{black};
	\end{tikzpicture}
	\end{subfigure}
\end{figure}

Let's consider the nonfree fundamental classes $[D]\in H^{1,1}(X)$, $[x],[y]\in H^{2,2}(X)$, and the free fundamental classes, $[C\sqcup \sigma C]_q, [C' \sqcup \sigma C']_q\in H^{1,q}(X)$ and \linebreak $[z\sqcup \sigma z]_q\in H^{2,q}(X)$. Note $[C\sqcup \sigma C]_q$ and $[C'\sqcup \sigma C']_q$ forget to different nonzero classes, so both are nonzero, and they are not equal. These two families of classes therefore give rise to the two $\Sigma^{1,0}A_0$ summands appearing. One can also check $[D]=\rho$ and $[x]+[y]=\rho^2=\rho\cdot[D]$ using arguments as in the previous examples. 

From Example \ref{freepoint}, $[z\sqcup \sigma z]_q=\frac{\theta}{\tau^{q}}[p]$ for $q\leq 0$. This gives the multiplicative relation 
	\[[C\sqcup \sigma C]_r\smile[C' \sqcup \sigma C']_s=[z \sqcup \sigma z]_{r+s}=\tfrac{\theta}{\tau^{r+s}}[p]
	\] 
for $r+s\leq 0$. 
\end{example}

\section{Background on \texorpdfstring{$C_2$}{C2}-surfaces}\label{ch:backsurf}
In \cite{D2} all $C_2$-surfaces were classified up to equivariant isomorphism, and furthermore, a language was developed for describing the $C_2$-structure on a given equivariant surface. We review some of this language and the necessary parts of the classification. We then review the Bredon cohomology of all nonfree $C_2$-surfaces in $\underline{\Z/2}$-coefficients which was computed in \cite{Ha}. We begin with a few constructions. 

\begin{definition}\label{consumdef} 
Let $X$ be a nontrivial $C_2$-surface and $Y$ be a nonequivariant surface. We can form the \textbf{equivariant connected sum} of $X$ and $Y$ as follows. Let $Y'$ denote the space obtained by removing a small disk from $Y$. Let $D$ be a disk in $X$ that is disjoint from its conjugate disk $\sigma D$ and let $X'$ denote the space obtained by removing both of these disks. Choose an isomorphism $f:\partial Y' \to \partial D$. Then the equivariant connected sum is given by
\[\left(Y' \times \{0\}\right) \sqcup \left(Y' \times \{1\} \right)\sqcup X']/\sim\]
where $(y,0) \sim f(y)$ and $(y,1) \sim \sigma(f(y))$ for $y\in \partial Y'$. We denote this space by $X \#_2 Y$.
\end{definition}

\begin{rmk}\label{specialconnsum} There are two important examples of equivariant connected sums that warrant their own notation. These occur when $X$ is one of $S^{2,2}$ or $S^{2,1}$, the two nonfree, nontrivial equivariant spheres. We refer to such spaces as \textbf{doubling spaces}, and denote the space $X\#_2 Y$ by Doub$(Y,1:S^{1,1})$ or Doub$(Y,1:S^{1,0})$, respectively. The phrase ``doubling" is used to acknowledge that nonequivariantly $X\#_2Y$ is homeomorphic to $Y\# Y$. 
\end{rmk}

The next two constructions involve removing conjugate disks in order to attach an equivariant handle. There are two types of handles that can be attached. The first handle is given by $S^{1,1}\times D(\R^{1,1})$, where $D(\R^{1,1})$ is the unit disk in $\R^{1,1}$; we will refer to such a handle as an ``$S^{1,1}-$antitube". The second type of handle is given by $S^{1,0}\times D(\R^{1,1})$; we refer to this handle as an ``$S^{1,0}-$antitube". We give the precise definitions below.

\begin{definition} Let $X$ be a nontrivial $C_2$-surface. Form a new space denoted $X+[S^{1,0}-AT]$ as follows. Let $D$ be a disk contained in $X$ that is disjoint from its conjugate disk $\sigma D$. Remove both disks from $X$ and then attach an $S^{1,0}-$antitube. Whenever we construct such a space, we say we have done \textbf{$\mathbf{S^{1,0}-}$surgery}.
\end{definition}

We can similarly define \textbf{$\mathbf{S^{1,1}-}$surgery} by instead attaching an $S^{1,1}-$antitube. 

The last type of surgery we recall involves removing a disk isomorphic to $D(\R^{2,2})$ the unit disk in $\R^{2,2}$ and sewing in an equivariant M\"obius band. This equivariant M\"obius band can be formed as follows. Begin with the nonequivariant M\"obius bundle over $S^1$ and then define an action on the fibers by reflection. In other words, each fiber should be isomorphic to $\R^{1,1}$ (in particular, the zero section is fixed). If we now take the closed unit disk bundle, note the boundary is a copy of $S^1_a$, as is the boundary of the removed $D(\R^{2,2})$. An illustration of this M\"obius bundle is shown below. Conjugate points are indicated by matching symbols, while the fixed set is shown in {\textbf{\color{blue}{blue}}}. 
\begin{figure}[ht]
	\includegraphics[scale=0.6]{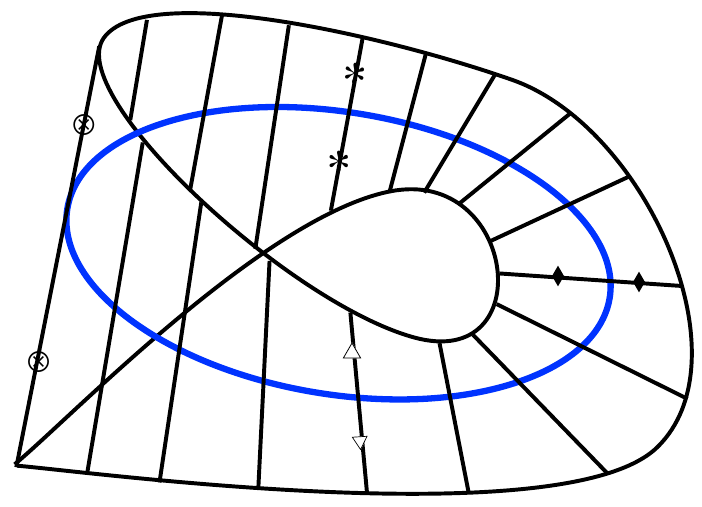}
\end{figure}

\begin{definition} 
Let $X$ be a nontrivial $C_2$-surface that contains an isolated fixed point $p$. Then there exists an open disk $p\in D\subset X$ such that $D\cong D(\R^{2,2})$. Remove $D$ and note $\partial D \cong S^{1}_a$. Now sew in a copy of the M\"obius band described above. We denote this new space by $X+[FM]$ and refer to this surgery as \textbf{$\mathbf{FM-}$surgery}. (Note ``FM" is an abbreviation for ``fixed point to M\"obius band".)
\end{definition}
\medskip 

\subsection{Invariants} 
When discussing $C_2$-surfaces, we will often need to refer to certain invariants. The notation for these invariants is given below.

\begin{notation} 
Let $X$ be a $C_2$-surface. Then
\begin{itemize}
	\item $F(X)$ denotes the number of isolated fixed points;
	\item $C(X)$ denotes the number of fixed circles; and
	\item $\beta(X)$ denotes the dimension of $H^{1}_{sing}(X;\Z/2)$ and will be referred to as the \emph{$\beta$-genus} of $X$.
\end{itemize}
When there is no ambiguity about which space is in discussion, we will just write $F$, $C$, and $\beta$ for the respective values. 
\end{notation}

The complete list of nonfree, nontrivial $C_2$-surfaces is given in \cite{D2}, but we only need the following takeaway for our discussion of fundamental classes.

\begin{theorem}\label{nonfreeclass} 
Let $X$ be a nonfree, nontrivial $C_2$-surface. If $X$ is not isomorphic to a doubling space or to an equivariant sphere, then $X$ can be obtained by doing $S^{1,0}-$, $S^{1,1}-$, or $FM-$ surgery to an equivariant space of lower $\beta$-genus. 
\end{theorem}
\medskip

\subsection{Bredon cohomology of nonfree $C_2$-surfaces}
We now state the Bredon cohomology of nonfree $C_2$-surfaces in constant $\underline{\Z/2}$-coefficients which was computed in \cite[Theorem 6.6]{Ha}. Note the answer depends only on the $\beta$-genus and the fixed set. 

\begin{theorem}\label{nonfreeanswer} 
Let $X$ be a nontrivial, nonfree $C_2$-surface. There are two cases for the $RO(C_2)$-graded Bredon cohomology of $X$ in \underline{$\Z/2$}-coefficients.
\begin{enumerate}
	\item[(i)] Suppose $C=0$. Then
		\[
		H^{\ast, \ast}(X; \underline{\Z/2})\cong \M_2 \oplus \left(\Sigma^{1,1} \M_2\right)^{\oplus F-2} \oplus \left(\Sigma^{1,0}A_0 \right)^{\oplus \frac{\beta-F}{2}+1} \oplus \Sigma^{2,2}\M_2.
		\]
	\item[(ii)] Suppose $C\neq 0$. Then 
		\begin{align*}
			H^{\ast, \ast}(X; \underline{\Z/2})\cong & ~\M_2 \oplus \left(\Sigma^{1,1} \M_2\right)^{\oplus F+C-1} \oplus  \left(\Sigma^{1,0} \M_2\right)^{\oplus C-1}\\
			&\oplus \left(\Sigma^{1,0}A_0 \right)^{\oplus \frac{\beta-F}{2}+1-C}\oplus \Sigma^{2,1}\M_2.
		\end{align*}
\end{enumerate}
\end{theorem}
\section{Fundamental classes for $C_2$-surfaces}\label{ch:funclassessurf}
It is straightforward to check the singular cohomology in $\Z/2$-coefficients of any surface is generated by fundamental classes. In the previous examples, we saw the analogous statement held for the Bredon cohomology of a handful of $C_2$-surfaces. In this section, we show, in fact, the Bredon cohomology of any $C_2$-surface is generated by fundamental classes.

\begin{notation} As before, the coefficients are understood to be $\underline{\Z/2}$ in this section. 
\end{notation}

We begin by defining the precise property we will be proving.

\begin{definition} 
Let $X$ be a $C_2$-manifold. Suppose there exist a (possibly empty) collection of nonfree equivariant submanifolds $Y_1,\dots Y_n$ and a (possibly empty) collection of free equivariant submanifolds $F_1,\dots,F_m$ such that the corresponding fundamental classes generate $H^{*,*}(X)$ as an $\M_2$-module, i.e.
	\[
	\M_2\{ [Y_{1}],\dots, [Y_{n}],[F_1]_q,\dots, [F_m]_q\colon q\in \Z  \} = H^{*,*}(X).
	\]
Then we say $H^{*,*}(X)$ is {\textbf{generated by fundamental classes}}.
\end{definition}

Our goal is to show if $X$ is any $C_2$-surface, then $H^{*,*}(X)$ is generated by fundamental classes. We begin with free surfaces, spheres, and doubling spaces, and then consider surfaces obtained by doing surgery to such spaces and apply Theorem \ref{nonfreeclass}.

\begin{warning}
One might hope the submanifolds $Y_i$ can always be chosen to be fixed. This has been the case in our examples thus far, but it is not true in general. For example, consider the torus with the rotation action that has four fixed points (this can also be described as $S^{2,2}+[S^{1,1}-AT]$). By Theorem \ref{nonfreeanswer} the cohomology of this space is
	\[
	\M_2 \oplus \left(\Sigma^{1,1}\M_2\right)^{\oplus 2} \oplus \Sigma^{2,2}\M_2.
	\]
We claim $\theta\cdot[F_j]=0$ for any free submanifold $F_j$, and thus no free submanifold can generate a free summand. To see this, recall $[F_j]$ is defined as the image of the Thom class of the normal bundle, and the Thom class of a bundle over a free space has $\theta$-torsion by Theorem \ref{thomfree} and Lemma \ref{freestructure}. On the other hand, the fixed set has codimension $2$, and thus no fixed submanifold can generate the free summands in topological degree $1$. In this case, we need classes that are neither free nor fixed in order to generate the cohomology. 
\end{warning}

\begin{lemma}\label{freefun} 
Suppose $X$ is a free $C_2$-surface. Then $H^{*,*}(X)$ is generated by fundamental classes.
\end{lemma}
\begin{proof} 
Since $X$ is free, the quotient map $\pi:X\to X/C_2$ is a smooth map of surfaces. Now $H^{*}_{sing}(X/C_2)$ is generated by fundamental classes of submanifolds, and let $Y_1,\dots,Y_j$ denote these submanifolds of $X/C_2$. By Theorem \ref{funnat}, $\pi^*[Y_i]=[\pi^{-1}(Y_i)]_0$ and we claim the classes for $\pi^{-1}(Y_i)$ generate the cohomology of $X$. Indeed, by Remark \ref{freeboriso},
	\[
	H^{*,*}_{Bor}(X) \cong H^{*}_{sing}(X/C_2) \underset{\Z/2[u]}{\otimes} \Z/2[\rho, \tau^{\pm 1}],
	\]
and so the Bredon cohomology of $X$ is determined by the singular cohomology of $X/C_2$. This completes the proof.
\end{proof}

\begin{lemma}\label{spherefun} 
Suppose $X$ is a $C_2$-sphere. Then $H^{*,*}(X)$ is generated by fundamental classes.
\end{lemma}
\begin{proof} There are exactly four $C_2$-spheres up to equivariant isomorphism: $S^2_a$, $S^{2,0}$, $S^{2,1}$, and $S^{2,2}$. Note $S^2_a$ was handled in the above theorem. When $X$ is $S^{2,0}$, $S^{2,1}$, or $S^{2,2}$, we need only take the classes $[X]$ and $[p]$ where $p\in X$ is some fixed point. These classes will generate $H^{*,*}(X)$. 
\end{proof}

\begin{lemma}\label{doubfun} 
Suppose $X$ is a doubling space. Then $H^{*,*}(X)$ is generated by fundamental classes.
\end{lemma}
\begin{proof} There are two cases, either $X\cong S^{2,2}\#_2 Y$ or $X\cong S^{2,1}\#_2 Y$. The proof for both cases is very similar, so we only provide a proof for the former.

Using Theorem \ref{nonfreeanswer}, the cohomology of $S^{2,2}\#_2 Y$ is given by
	\[
	H^{*,*}(X) \cong \M_2 \oplus \left(\Sigma^{1,0}A_0\right)^{\oplus \beta(Y)} \oplus \Sigma^{2,2}\M_2.
	\]
The class $[X]$ will generate the $\M_2$-summand appearing in bidegree $(0,0)$, while $[p]$ where $p\in X$ is any fixed point is a generator for the $\M_2$-summand appearing in bidegree $(2,2)$. Thus we must find $\beta(Y)$ one-dimensional submanifolds whose classes will generate the $A_0$-summands appearing in topological degree one. 

Let $n=\beta(Y)$ and $Y_1,\dots, Y_n$ be one-dimensional submanifolds of $Y$ that generate $H^{1}_{sing}(Y)$. Note we can choose these submanifolds so that they do not intersect the neighborhood removed from $Y$ to form $X$. In $X$ we will have submanifolds of the form $Y_i\times C_2$ for each $i$. The classes $[Y_1]_{sing}, \dots, [Y_n]_{sing}$ form a basis for $H^1_{sing}(Y)$, so the classes $[Y_1\times C_2]_{sing}, \dots, [Y_n\times C_2]_{sing}$ are linearly independent in $H^{1}_{sing}(X)$. From the forgetful map, these classes must be linearly independent in $H^{1,0}(X)$, and hence they must generate the $A_0$-summands in topological degree one. This completes the proof. 
\end{proof}

Our next goal is to show this property holds for all nonfree surfaces. To do so, we make use of Theorem \ref{nonfreeclass}, which states if $X$ is a nontrivial $C_2$-surface that is not free, not isomorphic to a sphere, and not isomorphic to a doubling space, then up to equivariant isomorphism, $X$ can be constructed by doing $S^{1,1}-$, $S^{1,0}-$, or $FM-$surgery to a $C_2$-surface of lower $\beta$-genus. We first prove a lemma that will be helpful in what follows. 

\begin{lemma}\label{includept} 
Let $X$ be a nonfree $C_2$-surface. Suppose $C\subset X$ is a one-dimensional submanifold such that $C\cong S^{1,1}$ and $[C]\in H^{1,1}(X)$. Let $q\in C$ be a fixed point. Then $\iota_q^*([C])=\rho$ where $\iota_q:\{q\} \hookrightarrow X$ is the inclusion map.
\end{lemma}
\begin{proof} 
Consider the map $i:X^{C_2}\hookrightarrow X$. Note $C^{C_2}$ consists of two isolated points; let $q$ and $q'$ denote these points. By Corollary \ref{funfix}, $i^*[C] = \rho[q]+\rho[q']$. The lemma then follows by noting $\iota_q$ can be factored $\{q\}\hookrightarrow X^{C_2}\hookrightarrow X$.
\end{proof}

We now investigate how doing surgery introduces new fundamental classes. 

\begin{lemma}\label{s10class} 
Let $Y$ be a closed $C_2$-surface such that $H^{*,*}(Y)$ is generated by fundamental classes. If $X=Y+[S^{1,0}-AT]$, then $H^{*,*}(X)$ is also generated by fundamental classes. 
\end{lemma}
\begin{proof} The surface $X$ contains a fixed circle, so by Theorem \ref{nonfreeanswer}
	\[
	H^{*,*}(X)\cong \M_2 \oplus \Sigma^{2,1}\M_2 \oplus \dots
	\]
where the remaining modules appearing in the decomposition depend on $Y$. As usual, the class $[X]$ will generate the free summand appearing in bidegree $(0,0)$ while if $x\in X$ is any point contained in a fixed circle, the class $[x]$ will generate a free summand in bidegree $(2,1)$. We just need to find various one-dimensional submanifolds of $X$ that generate the remainder of the cohomology. The procedure to find such submanifolds will depend on the isomorphism type of $Y$.

Suppose $Y_1,\dots,Y_n$ are nonfree one-dimensional equivariant submanifolds of $Y$ and $F_1,\dots F_m$ are free one-dimensional submanifolds of $Y$ whose fundamental classes together with $[Y]$ and $[y]$ generate $H^{*,*}(Y)$. Let $U_1,\dots,U_n, V_1,\dots, V_m$ be equivariant tubular neighborhoods of $Y_1,\dots,Y_n, F_1, \dots F_m$, respectively. In order to form $X=Y+[S^{1,0}-AT]$, we must remove disjoint conjugate disks $D$, $\sigma D$ from $Y$, and we may assume without loss of generality that these disks and the tubular neighborhoods are chosen in a way such that $(D\cup \sigma D) \cap U_i$ and $(D\cup \sigma D)\cap V_j$ are empty for all $i,j$.

Let $Y'$ be the space obtained by removing $D\sqcup\sigma D$ from $Y$. We would like to relate the cohomology of $Y$ to the cohomology of $X$, and we will use the cohomology of the space $Y'$ as a stepping stone from $H^{*,*}(Y)$ to $H^{*,*}(X)$. 

Observe $Y$ and $X$ can be realized as the homotopy pushouts of the following diagrams, respectively:
\begin{center}
	\begin{tikzcd}
	\partial \left(D\sqcup \sigma D\right) \arrow[r,"\iota"] \arrow[d,"\pi"] & Y'\\
	C_2 &
	\end{tikzcd}\hspace{.5in}
	\begin{tikzcd}
	\partial \left(D\sqcup \sigma D\right) \arrow[r,"\iota"] \arrow[d,"\pi"] & Y'\\
	S^{1,0}  &
	\end{tikzcd}
\end{center}
Using the diagram on the left, for each $q$ we have a long exact sequence:
\vspace{-0.4cm}
\begin{center}
	\begin{tikzcd}[column sep=small]
	H^{0,q}(Y')\oplus H^{0,q}(C_2)\arrow[r]& H^{0,q}(\partial(D \sqcup \sigma D))\arrow[r]& H^{1,q}(Y) \arrow[r]&H^{1,q}(Y')\oplus H^{1,q}(C_2)
	\end{tikzcd}
\end{center}
The map $\pi^*:H^{0,q}(C_2)\to H^{0,q}(\partial(D\sqcup\sigma D))$ is an isomorphism because on the level of spaces
\[C_2\hookrightarrow \partial(D\sqcup\sigma D)\overset{\pi}{\to} C_2\]
is the identity, so $\pi^*$ is an injective map from $\Z/2$ to $\Z/2$. Thus the leftmost map is surjective, and the rightmost map is injective by exactness. Though $H^{1,q}(C_2)=0$ so the map $H^{1,q}(Y)\to H^{1,q}(Y')$ is injective. The inclusion $Y'\hookrightarrow X$ induces a map $H^{1,q}(X)\to H^{1,q}(Y')$; this map is often not injective, but for each submanifold $C$ from our list, we have the following commutative diagram:
\begin{center}
	\begin{tikzcd}
	H^{1,q}(Y,Y-C) \arrow[r,"\cong"] \arrow[d]& H^{1,q}(Y',Y'-C)\arrow[d] &H^{1,q}(X,X-C) \arrow[l,"\cong" above]\arrow[d]\\
	H^{1,q}(Y)\arrow[r, hook] & H^{1,q}(Y') & H^{1,q}(X) \arrow[l]
	\end{tikzcd}
\end{center}
Note the top horizontal maps are isomorphisms due to excision: all three of these groups are isomorphic to $H^{1,q}(U,U-C)$ where $U$ is the chosen tubular neighborhood of $C$ (note this is why we chose the disks and neighborhoods to be disjoint). The bottom left horizontal map is injective from the above discussion.

In particular, this commutative diagram holds for $C=Y_1,\dots, Y_n$ and $C=F_1,\dots, F_m$. Hence, the image of each of the classes $[Y_i]$ and $[F_j]_k$ in $H^{1,*}(X)$ under the right horizontal map is equal to the image of the respective class $[Y_i]$ or $[F_j]_k$ in $H^{1,*}(Y)$ under the left horizontal map. The injectivity of the left map shows the fundamental classes $[Y_{i}], [F_j]_k$ inherit no new relations in the cohomology of $X$ that were not present in the cohomology of $Y$. Intuitively, this is unsurprising: attaching a handle should not introduce dependence relations, and the above formalizes this intuition. 

There are three cases for how the cohomology of $X$ differs from the cohomology of $Y$. First, suppose $Y$ already contains a fixed oval. Then by Theorem \ref{nonfreeanswer}
\begin{align}\label{eq:yy}
	H^{*,*}(Y) \cong &~\M_2 \oplus \left(\Sigma^{1,0}\M_2\right)^{\oplus C(Y)-1} \oplus \left( \Sigma^{1,1}\M_2\right)^{\oplus F(Y)+C(Y)-1} \\
	&\oplus \left( \Sigma^{1,0} A_0\right)^{\frac{\beta(Y)+2-(F(Y)+2C(Y))}{2}} \oplus \Sigma^{2,1}\M_2 \nonumber
\end{align}
while
\begin{align*}
	H^{*,*}(X) \cong &~\M_2 \oplus \left(\Sigma^{1,0}\M_2\right)^{\oplus C(X)-1} \oplus \left( \Sigma^{1,1}\M_2\right)^{\oplus F(X)+C(X)-1} \\
	&\oplus \left( \Sigma^{1,0} A_0\right)^{\frac{\beta(X)+2-(F(X)+2C(X))}{2}} \oplus \Sigma^{2,1}\M_2 
\end{align*}
Since $X=Y+[S^{1,0}-AT]$, we have the following relations
	\[
	F(X)=F(Y),~~C(X)=C(Y)+1,~~\beta(X)=\beta(Y)+2
	\]
that show
\begin{equation}\label{eq:xx}
	H^{*,*}(X) \cong H^{*,*}(Y) \oplus \Sigma^{1,1}\M_2 \oplus \Sigma^{1,0}\M_2.
\end{equation}
Recall the classes $[X]$ and $[x]$ generate free summands in topological degrees zero and two, respectively, while the classes $[Y_i],[F_j]_q$ generate any summands appearing in topological dimension one coming from $H^{*,*}(Y)$ by the discussion above. Thus it suffices to find two new fundamental classes in $H^{1,*}(X)$.  

There is an obvious choice for one of the submanifolds, namely the fixed circle contained in the attached handle. Let $C_1$ be this circle and note $[C_1]\in H^{1,1}(X)$. For the other submanifold, let $p\in C_1$ be a fixed point, and choose a point $s$ contained in another fixed circle. It follows from \cite[Corollary A.2]{D2} that we can construct a path $\gamma$ from $p$ to $s$ such that
	\[
	C_\gamma :=\im(\gamma) \cup \im(\sigma \gamma) \cong S^{1,1}
	\]
and such that $C_\gamma$ and $C_1$ intersect at the single point $p$. Let $U$ be a tubular neighborhood of $C_\gamma$. Over each fixed point the normal fiber is a fixed interval, so $[C_\gamma]\in H^{1,0}(X)$. The two classes are in the correct bidegrees; we next show they are linearly independent from the classes coming from $Y$. 

We begin with $[C_\gamma]$. By construction $C_1$ and $C_\gamma$ intersect at a single fixed point, so
	\[
	[C_1][C_\gamma] = [p]\neq 0.
	\]
On the other hand, $C_1$ does not intersect any of the other submanifolds $Y_i,F_j$, and so for any $\M_2$-linear combination of these fundamental classes
	\[
	[C_1]\cdot \left(\Sigma_i a_i[Y_i] + \Sigma_j b_j [F_j]_q\right) = 0.
	\] 
Hence, it must be that $[C_\gamma]$ is not in the $\M_2$-span of these classes. By the isomorphism in \ref{eq:yy}, the fact that $H^{*,*}(Y)$ is generated by fundamental classes, and consideration of degrees, it must be that
\begin{align*}
	\dim\left(H^{1,0}(Y)\right)=&\dim\left(\M_2\{ [Y_{i}],[F_j]_q, [Y], [y] \} \cap H^{1,0}(Y)\right) \\
	=&\dim\left(\M_2\{ [Y_{i}],[F_j]_q  \} \cap H^{1,0}(Y)\right).
\end{align*}
We have already remarked that
\begin{align*}
	\dim\left(\M_2\cdot\{[Y_{i}],[F_j]_q\} \cap H^{1,0}(Y)\right) = \dim\left(\M_2\cdot\{[Y_{i}],[F_j]_q\} \cap H^{1,0}(X)\right).
\end{align*}
By the isomorphism in \ref{eq:xx}, 
\begin{align*}
	\dim\left(H^{1,0}(Y)\right) = \dim\left(H^{1,0}(X)\right)-1.
\end{align*}
We just proved $[C_\gamma]$ is not in the $\M_2$-span of the classes $[Y_i]$, $[F_j]_q$, so by dimensions it must now follow that
\begin{align*}
	\M_2\cdot\{[C_\gamma],[Y_{i}],[F_j]_q\} \cap H^{1,0}(X) = H^{1,0}(X).
\end{align*}
Thus any generator of the new summand $\Sigma^{1,0}\M_2$ is a linear combination of fundamental classes.

Returning to $C_1$, we can apply Lemma \ref{includept} to see $\iota_p^{*}([C_1])$ is nonzero. Since $\{p\}$ does not intersect any of the submanifolds $Y_i,F_j$, $\iota_p^*([Y_i])=\iota_p^*([F_j]_q)=0$. For degree reasons, $\iota_p^*([C_\gamma])$ is also zero, and thus any $\M_2$-combination of these classes must be in the kernel of $\iota_p$. We conclude $[C_1]$ is not in the $\M_2$-span of any of these classes. Again using our isomorphisms and degree arguments we can say
\begin{align*}
	&\dim\left(\M_2\{ [C_1],[C_\gamma], [Y_{i}],[F_j]_q,[X]\} \cap H^{1,1}(X)\right) = \dim\left(H^{1,1}(X)\right).
\end{align*}
We conclude any generator of the new summand $\Sigma^{1,1}\M_2$ is in this span. Thus
	\[
	H^{*,*}(X)  = \M_2\{[C_1], [C_\gamma], [Y_{i}], [F_j]_q, [X], [x] \} 
	\]
as desired. This completes the proof in the case that $Y$ contains a fixed oval.\medskip

Next suppose the fixed set of $Y$ contains only isolated fixed points. Then by Theorem \ref{nonfreeanswer}
\begin{align*}
	H^{*,*}(Y) \cong &~\M_2  \oplus \left( \Sigma^{1,1}\M_2\right)^{\oplus F(Y)-2} \oplus \left( \Sigma^{1,0} A_0\right)^{\frac{\beta(Y)+2-F(Y)}{2}} \oplus \Sigma^{2,2}\M_2 
\end{align*}
while
\begin{align*}
	H^{*,*}(X) \cong &~\M_2 \oplus \left( \Sigma^{1,1}\M_2\right)^{\oplus F(X)+C(X)-1} \\
	&\oplus \left( \Sigma^{1,0} A_0\right)^{\frac{\beta(X)+2-(F(X)+2C(X))}{2}} \oplus \Sigma^{2,1}\M_2 
\end{align*}
We have the relations
	\[
	F(X)=F(Y),~~C(X)=1, C(Y)=0,~~\beta(X)=\beta(Y)+2.
	\]
Hence the number of $\Sigma^{1,0}\M_2$ and $\Sigma^{1,0}A_0$ summands match, and we just need to find two new classes that generate the two additional $\Sigma^{1,1}\M_2$-summands.

As in the previous case, we need to find two new fundamental classes. Again let $C_1$ be the attached (and now only) fixed oval, noting $[C_1]\in H^{1,1}(X)$. Let $\gamma$ be a path from a fixed point $p\in C_1$ to an isolated fixed point $s$, choosing the path $\gamma$ such that 
	\[
	C_\gamma :=\im(\gamma) \cup \im(\sigma \gamma) \cong S^{1,1}.
	\] 
and such that $C_\gamma$ and $C_1$ only intersect at a single point. Let $U$ be a tubular neighborhood of the circle $C_\gamma$, and note that the fiber over $p$ is isomorphic to $\R^{1,0}$, while the fiber over $s$ is isomorphic to $\R^{1,1}$, so it must be that $[C_\gamma]\in H^{1,1}(X)$.

We can do the same tricks as before to conclude $[C_1]$ and $[C_\gamma]$ generate the rest of the cohomology of $X$. Namely, their nontrivial product shows $[C_\gamma]$ is not an $\M_2$-combination of our current classes. To see $[C_1]$ is not in the span of the classes given by $Y_i,F_j,C_\gamma$, choose another point on the fixed oval $p'\neq p$ and use the map $\iota_{p'}$. This will complete the proof in this case.\medskip

Lastly, suppose $Y$ is a free $C_2$-surface. Then by the computations in \cite{Ha}, ignoring the action of $\rho$, 
\begin{align*}
	H^{*,*}(Y) \cong &~A_0  \oplus \left( \Sigma^{1,0} A_0\right)^{\frac{\beta(Y)+2}{2}} \oplus \Sigma^{2,1}A_0
\end{align*}
while by Theorem \ref{nonfreeanswer}
\begin{align*}
	H^{*,*}(X) \cong &~\M_2 \oplus\left( \Sigma^{1,0} A_0\right)^{\frac{\beta(X)+2-(F(X)+2C(X))}{2}} \oplus \Sigma^{2,1}\M_2 
\end{align*}
We now have the relations
	\[
	F(X)=F(Y)=0,~~C(X)=1, C(Y)=0~~\beta(X)=\beta(Y)+2.
	\]
The number of summands generated in topological dimension one is the same in $H^{*,*}(X)$ as in $H^{*,*}(Y)$. Thus, by adding $[X]$ and $[x]$ to our list of classes $[Y_i],[F_j]_q$, we will have found a collection of fundamental classes that generate the cohomology of $X$.
\end{proof}

\begin{lemma}\label{s11class} Let $Y$ be a nontrivial $C_2$-surface such that $H^{*,*}(Y)$ is generated by fundamental classes. If $X=Y+[S^{1,1}-AT]$, then $H^{*,*}(X)$ is also generated by fundamental classes. 
\end{lemma}
\begin{proof} 
The proof is similar to that of the previous lemma, so we just provide an outline. Instead of having a fixed circle in the attached handle, we have a circle $C_1\cong S^{1,1}$. The fundamental class for $C_1$ is in bidegree $(1,1)$. In the case where $Y$ has an isolated fixed point, the other class will be given by a circle $C_\gamma$ where $\gamma$ is a path from a fixed point on $C_1$ to this other isolated fixed point; this will give another class in bidegree $(1,1)$. In the case where $Y$ only has fixed ovals, the other class will given by a circle $C_\gamma$ where $\gamma$ is a path from a fixed point on $C_1$ to a point on a fixed oval. Lastly, in the case where $Y$ is free, no other classes besides $[X]$ and $[x]$ will be needed.
\end{proof}

We are now ready to prove the main theorem of this section.

\begin{theorem} Let $X$ be a $C_2$-surface. Then $H^{*,*}(X;\underline{\Z/2})$ is generated by fundamental classes of equivariant submanifolds.
\end{theorem}
\begin{proof} If $X$ is trivial, then by Lemma \ref{trivial}
	\[
	H^{*,*}(X)\cong \M_2\otimes_{\Z/2} H^{*}_{sing}(X).
	\]
Since $H^{*}_{sing}(X)$ is generated by fundamental classes, it immediately follows that $H^{*,*}(X)$ is generated by fundamental classes.

Assume $X$ is nontrivial. We proceed by induction on the $\beta$-genus of $X$. If the $\beta$-genus is zero, then we are done by Lemma \ref{spherefun}. For the inductive hypothesis, let $k\geq 1$ and assume the statement holds for all surfaces of $\beta$-genus less than $k$. 

Let $X$ be a surface of $\beta$-genus $k$. If $X$ is a free $C_2$-surface or a doubling space, then we are done by Lemmas \ref{freefun} and \ref{doubfun}. Thus suppose $X$ is nonfree and not a doubling space. By Theorem \ref{nonfreeclass}, there are three cases for $X$: the surface $X$ is isomorphic to a space given by doing $S^{1,1}-$, $S^{1,0}-$, or $FM-$ surgery to a surface of lower $\beta$-genus.

Combining the first and second cases, suppose $X\cong Y+[S^{1,\epsilon}-AT]$ where $\epsilon\in \{0,1\}$. Note $\beta(Y)=\beta(X)-2$, so we can apply the inductive hypothesis to conclude $H^{*,*}(Y)$ is generated by fundamental classes. We are then done after applying either Lemma \ref{s10class} or Lemma \ref{s11class}. 

The remaining case is $X\cong Y+[FM]$. Let $M$ denote a closed neighborhood of the attached M\"obius band and observe $X/M\cong Y$. Thus in this case we have a map $\pi:X\to Y$ that appears in the cofiber sequence
	\[
	X\to Y \to \Sigma^{1,0} M.
	\]
Note $M\simeq S^{1,0}$ and so it follows from the associated long exact sequence that the map $\pi^{*}:H^{1,q}(Y) \to H^{1,q}(X)$ is injective if $q\geq -2$. Since $\beta(Y)<\beta(X)$, the inductive hypothesis implies $H^{*,*}(Y)$ is generated by fundamental classes. Let $Y_1,\dots, Y_n$ be the nonfree and $F_1,\dots, F_n$ be the free the one-dimensional submanifolds of $Y$ whose fundamental classes along with $[Y]$ and $[y]$ generate $H^{*,*}(Y)$. We can choose these submanifolds so that $\pi^{-1}(Y_i)$ and $\pi^{-1}(F_j)$ are submanifolds of $X$. By Theorem \ref{funnat}, $\pi^*[Y_i]=[\pi^{-1}(Y_i)]$ and similarly $\pi^*[F_j]_q=[\pi^{-1}(F_j)]_q$. The injectivity of $\pi^*$ implies these classes inherit no new relations in $H^{*,*}(X)$. Similar to the proof of Lemma \ref{s11class}, our goal is to find new submanifolds to generate the remainder of $H^{*,*}(X)$. There are two subcases based on the fixed set of $Y$.

First suppose $Y^{C_2}$ contains a fixed circle. Then by Theorem \ref{nonfreeanswer},
	\[
	H^{*,*}(X) \cong H^{*,*}(Y) \oplus \Sigma^{1,0}\M_2,
	\]
so we just need to find a single one-dimensional submanifold. Let $C$ denote the circle in the attached M\"obius band. Then $[C]$ gives a class in bidegree $(1,1)$. As in Lemma \ref{s11class}, construct a circle $C_\gamma$ from a path that travels from a point on $C$ to a point on another fixed circle. Then $[C_\gamma]\in H^{1,0}(X)$. From the intersection product, $[C]\cdot [C_\gamma]\neq 0$ while $[C]\cdot [Y_i]=0$ and $[C]\cdot [F_j]_q=0$ for all $i$, $j$, and $q$. Thus $[C_\gamma]$ must be linearly independent from the classes already in $H^{1,*}(X)$. We conclude the classes $[\pi^{-1}(Y_i)]$, $[\pi^{-1}(F_j)]$, $[C_\gamma]$, $[X]$, and $[p]$ where $p$ is any point on a fixed circle will generate $H^{*,*}(X)$. 

Next suppose $Y$ does not have any fixed circles. The fixed set of $Y$ must contain at least one fixed point in order to do $FM-$surgery. By running the long exact sequence associated to the cofiber sequence $M\hookrightarrow X \to Y$ (or by Theorem \ref{nonfreeanswer}), we see $H^{*,*}(X)$ has the same number and type of summands generated in topological degree one as $H^{*,*}(Y)$ except it has an additional summand of the form $\Sigma^{1,1}\M_2$. Let $p$ be a fixed point on $C$. In this case, the classes $[\pi^{-1}(Y_i)]$, $[\pi^{-1}(F_j)]_q$, $[C]$, $[X]$, and $[p]$ will generate $H^{*,*}(X)$. 

We have exhausted all cases, and we conclude $H^{*,*}(X)$ is generated by fundamental classes for all $C_2$-surfaces. 
\end{proof}
\appendix
\section{An algebraic property of $\M_2$-modules}\label{ch:algproof}
In this appendix, we state and prove a theorem about maps between nice $\M_2$-modules that are isomorphisms in a certain range. This theorem will imply property (vi) of Theorem \ref{thomnonfree}. 

\subsection*{Notation and terminology}
We say an $\M_2$-module is ``nice" if it is a direct sum of finitely many copies of shifted free modules and shifted copies of $A_r=\tau^{-1}\M_2/(\rho^{r+1})$ for various values of $r$, and furthermore, if all shifts are given by actual representations, i.e. the shifts are given by $(p,q)$ where $p\geq q \geq 0$. We will refer to the $A_r$-summands as ``antipodal summands". Given an antipodal summand of the form $\Sigma^{s,0}A_r$, we can associate the tuple $(s;r)$; note an antipodal summand with tuple $(s;r)$ begins in topological dimension $s$ and ends in topological dimension $(s+r)$. Given an $\M_2$-module $V$ and an element $v\in V$ we will write $wt(v)$ for the weight of $v$. When considering a single bidegree, we will write $V^{f,g}$ for the elements of $V$ in bidegree $(f,g)$.

In the proof, we will consider certain quotients, submodules, and localizations of nice $\M_2$-modules in order to detect the properties of free versus antipodal summands. For an $\M_2$-module $M$ let $T(M)=\{m\in M\colon \rho^im=0 \text{ for some } i\}$. If $M$ is a nice module, note $T(M)$ consists of the antipodal summands and the bottom cones of free summands. We can then consider the quotient $M/T(M)$ which is isomorphic to a direct sum of top cones, one for each free summand in $M$. We will denote such quotients by $\tilde{M}$. If we further quotient to form $\tilde{M}/\im(\rho)$, we obtain a module isomorphic to a direct sum of shifts of the module $\Z/2[\tau]$, one for each free summand. We can also consider the localization $\tau^{-1}T(M)$ which is isomorphic to the antipodal summands of $M$.

We begin by proving a lemma about $\Z/2[t]$-modules which will be useful when considering the quotient $\tilde{M}/\im(\rho)$.

\begin{lemma}\label{poly}
Consider the graded polynomial ring $R=\Z/2[t]$ where $|t|=1$. Let $M$ be a finitely generated, free $R$-module with $R$-basis $\{\alpha_1,\dots, \alpha_m\}$. Suppose $N$ is another finitely generated, free $R$-module, and $\phi:M\to N$ is a degree $q$ map such that $\phi:M^g\to N^{g+q}$ is an isomorphism whenever $g\geq g_0$ for some integer $g_0$. Then there exists an $R$-basis $\{\beta_1,\dots,\beta_m\}$ for $N$ such that $|\alpha_i|+q\geq |\beta_i|$.  
\end{lemma}
\begin{proof}
There exist integers $j_1,\dots, j_m$ such that $g=|t^{j_i}\alpha_i|$ is larger than $g_0$ and constant for all $i$. The elements $t^{j_i}\alpha_i$ form a linearly independent set in $M^g$ and thus the images $\phi(t^{j_i}\alpha_i)$ form a linearly independent set in $N^{g+q}$. This implies there are at least $m$ free summands in $N$. We can similarly use $\phi^{-1}$ in the range it exists to show there can be at most $m$ free summands in $N$. We conclude $N$ has an $R$-basis consisting of $m$ elements.

We proceed by induction on $m$ to show there is a basis $\{\beta_1,\dots,\beta_m\}$ for $N$ such that $|\alpha_i|+q\geq |\beta_i|$. This is clear if the bases consist of exactly one element since otherwise $\phi$ would be zero. For the inductive hypothesis, suppose we can find such a basis whenever we have a map that is an isomorphism in sufficiently high degrees between finitely generated, free $R$-modules of rank $m-1$. Choose some $R$-basis $\{b_1,\dots, b_m\}$ for $N$ and suppose $\phi(\alpha_1)=\sum_{i}\epsilon_{i}t^{j_i}b_i$ where $\epsilon_i\in \Z/2$. Let $b_k$ be a basis element of maximal degree such that the coefficient $\epsilon_i$ is nonzero. Reorder the set so that this element is now $b_1$. We can factor out powers of $t$ to obtain
	\[
	\phi(\alpha_1)= t^{j_1}\left(b_1+\sum_{i>1} \epsilon_it^{j_i-j_1}b_i\right).
	\] 
Let $\beta_1=b_1+\sum_{i>1} \epsilon_it^{j_i-j_1}b_i$ and note the set $\{\beta_1, b_2,\dots, b_m\}$ is still an $R$-basis for $N$. Furthermore, we obtain a map
	\[
	\langle \alpha_2,\dots, \alpha_m\rangle \overset{\phi}{\longrightarrow}N \twoheadrightarrow N/\langle \beta_1\rangle
	\]
that is still an isomorphism in the desired range. By the inductive hypothesis, there exists an $R$-basis $\beta_2',\dots, \beta_m'$ for the quotient such that $|\beta_i'|\leq |\alpha_i|+q$. Let $\beta_i$ be a lift of $\beta_i'$ to $N$. Then $\beta_1,\dots, \beta_m$ will be a basis for $N$ such that $|\beta_i|\leq |\alpha_i|+q$. 
\end{proof}

\begin{theorem}\label{algproof}
Let $V$ and $W$ be two nice $\M_2$-modules such that $V$ has $c$ free summands generated in bidegrees $(k_i,\ell_i)$ and $d$ antipodal summands with tuples $(s_j;r_j)$. Suppose $\phi:V\to W$ is an $(n,q)$-degree $\M_2$-module map such that $\phi:V^{f,g}\to W^{f+n,g+q}$ is an isomorphism whenever $g\geq f$. Then $W$ has exactly $c$ free summands generated in bidegrees $(k_i+n,\ell_i')$ such that $\ell_i+q_i\geq\ell_i'\geq 0$ and exactly $d$ antipodal summands with tuples $(s_j+n;r_j)$. 
\end{theorem}
\begin{proof}
One readily checks the restrictions $\phi:T(V) \to T(W)$ and $\phi:\im(\rho)\to \im(\rho)$ are isomorphisms in the given range. The map $\phi$ then descends to a map
	\[
	\tilde\phi:(\tilde{V}/\im(\rho))^{f,g} \to (\tilde{W}/\im(\rho))^{f+n,g+q}
	\] 
that is still an isomorphism when $g\geq f$. The quotients $\tilde{V}/\im(\rho)$ and $\tilde{W}/\im(\rho)$ are isomorphic to a direct sum of shifts of the module $\Z/2[\tau]$, one for each free summand in $V$ and $W$, respectively. Explicitly
	\[
	\tilde{V}/\im(\rho)\cong \bigoplus_{i=1}^c \Sigma^{k_i,\ell_i}\Z/2[\tau].
	\]
Fix a topological dimension $k$ and consider $(\tilde{V}/\im(\rho))^{k,*}$ and $(\tilde{W}/\im(\rho))^{k+n,*}$. These are both finitely generated, free $\Z/2[\tau]$-modules and the map
	\[
	\tilde{\phi}:(\tilde{V}/\im(\rho))^{k,g}\to (\tilde{W}/\im(\rho))^{k+n,g+q}
	\]
is an isomorphism whenever $g\geq k$. By Lemma \ref{poly}, there are respective bases $\{\alpha_1,\dots, \alpha_m\}$ and $\{\beta_1,\dots, \beta_m\}$ for $(\tilde{V}/\im(\rho))^{k,*}$ and $(\tilde{W}/\im(\rho))^{k+n,*}$ such that $wt(\beta_i)\leq wt(\alpha_i)+q$. Lifts of the elements $\alpha_i$ and $\beta_i$ to $V$ and $W$ will generate the free summands in bidegrees $(k,wt(\alpha_i))$ and $(k+n,wt(\beta_i))$ in $V$ and $W$, respectively. This proves the statement about the free summands.

We next consider the antipodal summands. Consider the localization $\tau^{-1}T(V)$. Let $F=\Z/2[\tau,\tau^{-1}]$ and note $\tau^{-1}T(V)$ is a finitely generated $F[\rho]$-module. By the decomposition of $V$ into summands as an $\M_2$-module, we can say explicitly 
	\[
	\tau^{-1}T(V)\cong \bigoplus_{j=1}^d \Sigma^{s_j,0}F[\rho]/(\rho^{r_j+1}).
	\]
Similarly $\tau^{-1}T(W)$ is isomorphic to a direct sum of shifts of $F[\rho]/(\rho^{r+1})$, one for each antipodal summand of $W$. 

Note $\phi$ restricts to a map $\phi:T(V)\to T(W)$ which then localizes to give a map $\tau^{-1}\phi\colon \tau^{-1}T(V)\to \tau^{-1}T(W)$ of finitely generated $F[\rho]$-modules. Since $\phi$ was an isomorphism in bidegrees $(f,g)$ such that $g\geq f$, the map $\tau^{-1}\phi$ is guaranteed to be an isomorphism in these bidegrees. Though, now that $\tau$ is invertible, we see that $\tau^{-1}\phi$ is an isomorphism in all bidegrees. We conclude $\tau^{-1}T(V)\cong \tau^{-1}T(W)$ and $W$ must have exactly $d$ antipodal summands with tuples $(s_j+n;r_j)$.  
\end{proof}

\bibliographystyle{amsalpha}

\end{document}